\documentclass[11pt]{article}
\usepackage{color}
\usepackage{graphicx}
\usepackage{subfigure}
\usepackage{epsfig}
\usepackage{cite}
\usepackage{amsmath, amssymb}
\usepackage{fancybox}
\usepackage{listings}
\usepackage{verbatim}
\usepackage{url}
\usepackage{esint}
\usepackage{fancyhdr}
\usepackage{booktabs}
\usepackage{caption}
\newcommand{\red}{\textcolor[rgb]{1.00,0.00,0.00}}

\input{epsf.sty}
\newtheorem{theorem}{Theorem}[section]
\newtheorem{remark}{Remark}[section]
 
\newtheorem{proposition}{Proposition}[section] 
\newtheorem{corollary}{Corollary}[section] 
\newcommand{\remove}[1]{}
\newenvironment{proof}
{\begin{trivlist}
\item[\hspace{\labelsep}{\bf\noindent Proof. }]}
{$\hfill\Box$\end{trivlist}}

\title{\huge\bf 
Some results on the telegraph process confined by two non-standard boundaries\thanks{
To appear on {\bf Methodology and Computing in Applied Probability}. 
}} 
\author{
\bf Antonio Di Crescenzo\thanks{Dipartimento di Matematica,
Universit\`a di Salerno, Via Giovanni Paolo II n.\ 132, 84084
Fisciano (SA), Italy. e-mail: \texttt{adicrescenzo@unisa.it, bmartinucci@unisa.it, pparaggio@unisa.it}}
\and
\bf Barbara Martinucci$^\dag$ 
\and
\bf Paola Paraggio$^\dag$ 
\and
\bf Shelemyahu Zacks\thanks{Department of Mathematical Sciences, Binghamton University, 
Binghamton, NY 13902-6000, USA. e-mail: \texttt{shelly@math.binghamton.edu}}
}
\begin{document}
\maketitle

\begin{abstract}
We analyze the one-dimensional telegraph random process confined by two boundaries, 0 and $H>0$. 
The process experiences hard reflection at the boundaries (with random switching to full absorption). 
Namely, 
when the process hits the origin (the threshold $H$) it is either absorbed, with probability $\alpha$, or reflected 
upwards (downwards), with probability $1-\alpha$, for $0<\alpha<1$. We provide various results on the 
expected values of the renewal cycles and of the absorption time. The adopted approach is based on 
the analysis of the first-crossing times of a suitable compound Poisson process through linear boundaries.
Our analysis includes also some comparisons between suitable stopping times of the considered telegraph 
process and of the corresponding diffusion process obtained under the classical Kac's scaling conditions.

\end{abstract}
%
%
\section{Introduction}
%
The (integrated) telegraph process describes the random motion of a particle running upward or downward alternately with finite velocity. 
The particle motion reverses its direction at the random epochs of a homogeneous Poisson process. The probability law of the telegraph process   
is governed by a hyperbolic partial differential equation (the telegraph equation), which is  widely encountered in mathematical physics. 
Moreover, such process deserves interest in many other applied fields, such as finance,  biology and mathematical geosciences. 
\par
Since the seminal papers by Goldstein \cite{Goldstein} and Kac \cite{Kac}, many generalizations of the telegraph process have been proposed in the 
literature, such as the asymmetric telegraph process (cf.\ \cite{Beghin}, \cite{Lopez}), the generalized telegraph process 
(see, for instance, \cite{Bshouty}, \cite{CDCIMa2013}, \cite{DCMa2007}, \cite{DCMa2010}, \cite{Pogorui}, \cite{Zacks1}) or the jump-telegraph process (for example, \cite{DCMa2011}, \cite{DCMaIuZacks2013}, \cite{Lopez_Ratanov}, \cite{Ratanov3}, \cite{Ratanov2}). Other recent investigations have been devoted to suitable functionals of telegraph processes 
(cf. \cite{Kolesnik1}, \cite{Kolesnik2} and \cite{Martinucci}). 
See also Tilles and Petrovskii \cite{Tilles} for recent results on the reaction-telegraph process. 
A modern and exhaustive treatment of the one-dimensional telegraph process is provided  in the books by Kolesnik and Ratanov \cite{Ratanov} and Zacks \cite{Zacks}.
\par
The telegraph process and its numerous generalizations have  been largely studied in an unbounded space. Anyway, a certain interest to the effects of boundaries on the 
telegraph motion can be found in the literature. This interest mainly derives from the need to model physical systems in the presence of a variety 
of complex conditions (see, for instance, Ishimaru \cite{Ishimaru} for potential applications in the medical area of the telegraph equation in the presence of boundaries). 
The explicit distribution of the telegraph process subject to a reflecting or an absorbing barrier has been obtained in Orsingher \cite{Orsingher}, see also the related results given in Foong  and Kanno \cite{Foong}. Moreover, the treatment of a one-dimensional telegraph equation with either reflecting or partly reflecting boundary conditions has been discussed in Masoliver  {\em et al.}\ \cite{Masoliver}. See also the Chapter $3$ of Kolesnik and Ratanov \cite{Ratanov} for a wide review on the main results concerning 
the telegraph process in the presence of reflecting or absorbing boundaries.  
\par
The analysis of stochastic processes subject to 
non-standard boundaries is the object of investigation in various research areas. 
We are interested in processes that exhibit hard reflection at the boundaries (with random switching 
to full absorption). This behavior has been expressed in the past in terms of `elastic boundaries' 
(see Feller \cite{Feller} and Bharucha-Reid \cite{Bharucha}, for instance), however 
this terminology is used nowadays for different purposes. 
Some examples of papers concerning diffusion processes in the presence of 
such type of 
boundaries are provided by Giorno {\em et al.}\ \cite{Giorno2006} and Domin\'e \cite{Domine95}, \cite{Domine96}. 
See also Veestraeten \cite{Veestraeten} and Buonocore {\em et al.}\ \cite{Buonocore2002} 
for some applications of these processes in the areas of mathematical neurobiology and medicine. 
See also Bobrowski  \cite{Bobrowski} for a novel interpretation of some of the problems associated 
with telegraph processes and for a detailed description of the elastic Brownian motions.  
\par
The analysis of the telegraph process in the presence of an elastic boundary is a quite new research 
topic. Various results on the related absorption time and renewal cycles have been 
obtained by Di Crescenzo {\em et al.}\ \cite{DiCrescenzo}. A similar problem has been investigated 
by Smirnov \cite{Smirnov},  where a telegraph equation confined by two endpoints
is studied from an analytical point of view. 
\par
Stimulated by the above mentioned researches, in this paper we aim to investigate the problem of 
the telegraph process in the presence of two 
non-standard boundaries, say 0 and $H>0$.  
Specifically, if the process reaches 
any of the two boundaries, it can be absorbed, 
with probability $\alpha$, or reflected upwards (downwards), with probability $1-\alpha$, with $0<\alpha<1$.
As limit cases, one has pure reflecting behavior for $\alpha=0$ and pure absorbing behavior for $\alpha=1.$
Furthermore, we point out that the analysis of the absorption phenomenon deserves interest in many applied fields, 
such as chemistry and physics (see, for instance, Bohren and Huffman  \cite{Bohren} and L\"uders and Pohl \cite{Luders}). 
\par
Special attention is given to the explicit determination of the probabilities that the process hits a given 
boundary starting from each of the two endpoints. In this setting, the typical sample-paths of the motion  
can be analyzed under four different phases, according to the choices of the initial and the final endpoints.  
Moreover, we aim to determine the expected duration of the renewal cycles concerning the four phases, 
and the expected time till the absorption in one of the boundaries. 
In analogy with the approach exploited in Zacks \cite{Zacks}, in this paper we follow the methodology based 
on (i) the construction of a suitable compound Poisson process, say $Y(t)$, which provide useful information on the 
upward and downward random times of the particle motion, and (ii) the determination of  the mean first-passage-time 
of $Y(t)$ through linear boundaries.  
\par
Recalling the well-known 
limit behavior of the telegraph process under the Kac's scaling conditions leading to 
a Wiener process (c.f.\ Kolesnik and Ratanov \cite{Ratanov}), we also perform some comparisons 
between suitable stopping times of the telegraph process and of the corresponding diffusion process. 
\par
Let us now describe the content of the paper. In Section \ref{sec2} we  deal with the stochastic model 
concerning the telegraph motion in the presence of the 
specified boundaries 0 and $H$. Then, we introduce  the four phases, 
during which the particle starts from one endpoint and first hits one of the boundaries. 
In Section \ref{secStopping} we define the above-mentioned compound Poisson process $Y(t)$, 
and introduce the main quantities of interest related to the phases of the motion. Then, 
the first-hitting probabilities of the boundaries are determined in Section \ref{Sec4}, whereas 
Section \ref{Sec5} deals with the expected values of the first-passage-times. Finally, 
in Section \ref{Sec6} we study the expected time till absorption in one of the two boundaries. 

\section{The Stochastic Model}\label{sec2}
%
Let $\{X(t); t\geq 0\}$ be a one-dimensional telegraph process defined on a probability space
$(\Omega ,{\mathcal {F}},\mathbb {P} )$, and assume that it is confined between  
two  boundaries, one at $0$ and the other one at the level $H>0$. 
This process describes the motion of a particle over the state space $[0, H]$, starting from the origin at the initial time $t=0$. 
The particle moves upward and downward alternately with fixed velocity $1$. 
The motion proceeds upward for a positive random time $U_1$, then it moves downward for a positive random time $D_1$, and so on. 
We assume that $\{U_i\}_{i\in \mathbb{N}}$ and $\{D_i\}_{i\in \mathbb{N}}$ are independent sequences of i.i.d.\ random times. 
When the particle hits the origin or the threshold $H$ it is either absorbed, with probability $\alpha$, or reflected upwards (downwards), with probability $1-\alpha$, with $0<\alpha<1$. Specifically, if during a downward (upward) period, say $D_j$ ($U_j$), the particle reaches the origin (the threshold $H$) and is not absorbed, then instantaneously 
the motion restarts with positive (negative) velocity, according to an independent random time $U_{j+1}$ ($D_{j+1}$). A sample path of $X(t)$ is shown in Figure 1 where 
$D_j^*$ ($U_j^*$) denotes the downward (upward) random period $D_j$ ($U_j$) truncated by the occurrence of the visit at the origin (level $H$).
%
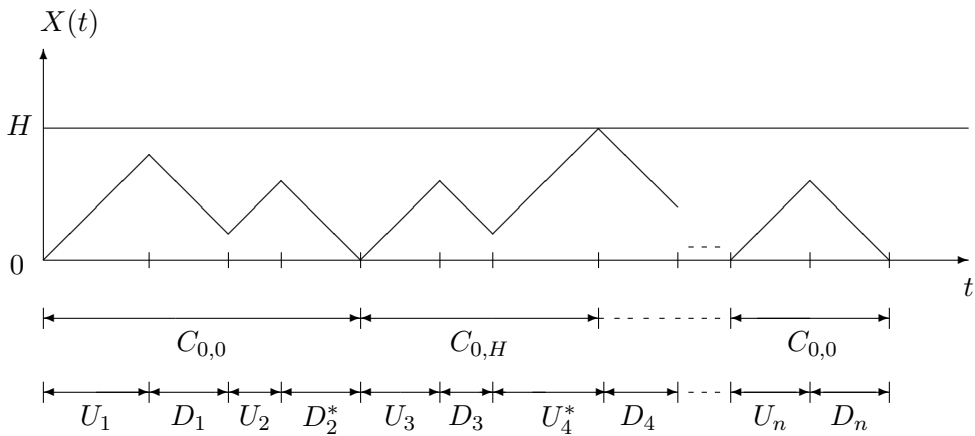
\begin{figure}[t]\label{fig:1}
\begin{center}
\hspace*{-1.8cm}
\begin{picture}(331,156) 
\put(20,70){\vector(1,0){350}} 
\put(20,120){\line(1,0){350}} 
\put(20,70){\vector(0,1){80}} 
\put(10,150){\makebox(40,15)[t]{$X(t)$}} 
\put(355,48){\makebox(30,15)[t]{$t$}} 
\put(0,57){\makebox(20,15)[t]{0}} 
\put(1,110){\makebox(20,15)[t]{$H$}} 
%
\put(20,70){\line(1,1){40}} 
\put(60,110){\line(1,-1){30}} 
\put(90,80){\line(1,1){20}} 
\put(110,100){\line(1,-1){30}} 
\put(140,70){\line(1,1){30}} 
\put(170,100){\line(1,-1){20}} 
%
%
\put(190,80){\line(1,1){40}}
%
%
\put(230,120){\line(1,-1){30}} 
\put(280,70){\line(1,1){30}} 
\put(310,100){\line(1,-1){30}} 
\put(60,67){\line(0,1){6}} 
\put(90,67){\line(0,1){6}} 
\put(110,67){\line(0,1){6}} 
\put(140,67){\line(0,1){6}} 
\put(170,67){\line(0,1){6}} 
\put(190,67){\line(0,1){6}} 
\put(230,67){\line(0,1){6}} 
\put(260,67){\line(0,1){6}} 
\put(280,67){\line(0,1){6}} 
\put(310,67){\line(0,1){6}} 
\put(340,67){\line(0,1){6}} 
\put(20,44){\line(0,1){8}} 
\put(140,44){\line(0,1){8}} 
\put(230,44){\line(0,1){8}} 
\put(280,44){\line(0,1){8}} 
\put(340,44){\line(0,1){8}} 
\put(70,48){\vector(-1,0){50}} 
\put(70,48){\vector(1,0){70}} 
\put(190,48){\vector(-1,0){50}} 
\put(190,48){\vector(1,0){40}} 
\put(310,48){\vector(-1,0){30}} 
\put(290,48){\vector(1,0){50}} 
\put(20,16){\line(0,1){8}} 
\put(60,16){\line(0,1){8}} 
\put(90,16){\line(0,1){8}} 
\put(110,16){\line(0,1){8}} 
\put(140,16){\line(0,1){8}} 
\put(170,16){\line(0,1){8}} 
\put(190,16){\line(0,1){8}} 
\put(232,16){\line(0,1){8}} 
\put(260,16){\line(0,1){8}} 
\put(280,16){\line(0,1){8}} 
\put(310,16){\line(0,1){8}} 
\put(340,16){\line(0,1){8}} 
\put(264,75){\line(1,0){2}} 
\put(269,75){\line(1,0){2}} 
\put(274,75){\line(1,0){2}} 
\put(264,20){\line(1,0){2}} 
\put(269,20){\line(1,0){2}} 
\put(274,20){\line(1,0){2}} 
\put(230,48){\line(1,0){2}} 
\put(236,48){\line(1,0){2}} 
\put(241,48){\line(1,0){2}} 
\put(247,48){\line(1,0){2}} 
\put(253,48){\line(1,0){2}} 
\put(259,48){\line(1,0){2}} 
\put(264,48){\line(1,0){2}} 
\put(269,48){\line(1,0){2}} 
\put(274,48){\line(1,0){2}} 
\put(60,20){\vector(-1,0){40}} 
\put(40,20){\vector(1,0){20}} 
\put(85,20){\vector(-1,0){25}} 
\put(85,20){\vector(1,0){5}} 
\put(105,20){\vector(-1,0){15}} 
\put(105,20){\vector(1,0){5}} 
\put(125,20){\vector(-1,0){15}} 
\put(125,20){\vector(1,0){15}} 
\put(155,20){\vector(-1,0){15}} 
\put(155,20){\vector(1,0){15}} 
\put(195,20){\vector(-1,0){5}} 
\put(185,20){\vector(1,0){5}} 
\put(210,20){\vector(-1,0){40}} 
\put(205,20){\vector(1,0){28}} 
\put(240,20){\vector(-1,0){8}} 
\put(240,20){\vector(1,0){20}} 
\put(305,20){\vector(-1,0){25}} 
\put(285,20){\vector(1,0){25}} 
\put(335,20){\vector(-1,0){25}} 
\put(315,20){\vector(1,0){25}} 
\put(55,28){\makebox(50,15)[t]{$C_{0,0}$}} 
\put(160,28){\makebox(50,15)[t]{$C_{0,H}$}} 
\put(286,28){\makebox(50,15)[t]{$C_{0,0}$}} 
\put(15,0){\makebox(50,15)[t]{$U_1$}} 
\put(50,0){\makebox(50,15)[t]{$D_1$}} 
\put(75,0){\makebox(50,15)[t]{$U_2$}} 
\put(100,0){\makebox(50,15)[t]{$D_2^*$}} 
\put(130,0){\makebox(50,15)[t]{$U_3$}} 
\put(155,0){\makebox(50,15)[t]{$D_3$}} 
\put(190,0){\makebox(50,15)[t]{$U_4^*$}} 
\put(220,0){\makebox(50,15)[t]{$D_4$}} 
\put(270,0){\makebox(50,15)[t]{$U_n$}} 
\put(300,0){\makebox(50,15)[t]{$D_n$}} 
\end{picture} 
\end{center}
\vspace{-0.3cm}
\caption{A sample-path of $X(t)$.}
\end{figure} 
The relevance of the random variables $D_j^*$ and $U_j^*$ will be clear in the proof of Theorem \ref{teor1} below. 
\par
Hence, we can recognize $4$ different phases of the motion of the particle:
\\
$-$ Phase 1: the particle starts from $0$ and returns again to $0$, without hitting $H$;\\
$-$ Phase 2: the particle starts from $0$ and hits $H$ before returning to $0$;\\
$-$ Phase 3: the particle starts from $H$ and returns again to $H$, without hitting $0$;\\
$-$ Phase 4: the particle starts from $H$ and hits $0$ before hitting $H$ again.\\
For the Phase $1$ (Phase $2$) we denote by $C_{0,0}$ ($C_{0,H}$) the random time between leaving the origin and arriving at the origin (at level $H$) without 
hitting $H$ (the origin). Similarly, for the Phase $3$ (Phase $4$), we denote by $C_{H,H}$ ($C_{H,0}$) the random time between leaving the level $H$ and arriving at 
$H$ (at the origin) without hitting the origin (the level $H$). These random times are called renewal cycles. Some examples of the different phases of the particle motion 
and of the corresponding renewal cycles are shown in Figure $2$.
\par
We remark that an affine transformation given by $\widetilde X_c(t):=c\, X(t)$, $t\geq 0$, easily leads to the 
case when the motion is characterized by an arbitrary constant velocity $c>0$. Hence, 
various results concerning velocity 1 can be easily extended to such more general case. 
\par
In the following sections an important role will be played by various useful stopping times, 
which will be assumed to be defined over the natural filtration of the relevant processes. 
\begin{figure}[t]
	\centering
	\subfigure[]{\includegraphics[scale=0.72]{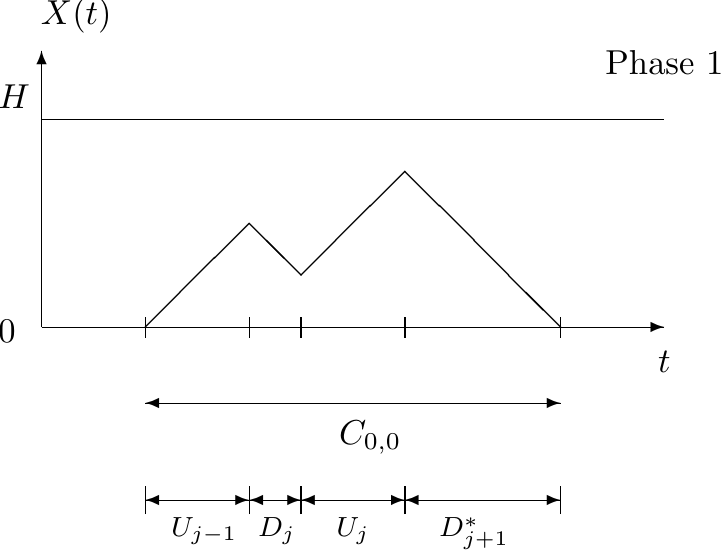}}\quad
	\subfigure[]{\includegraphics[scale=0.72]{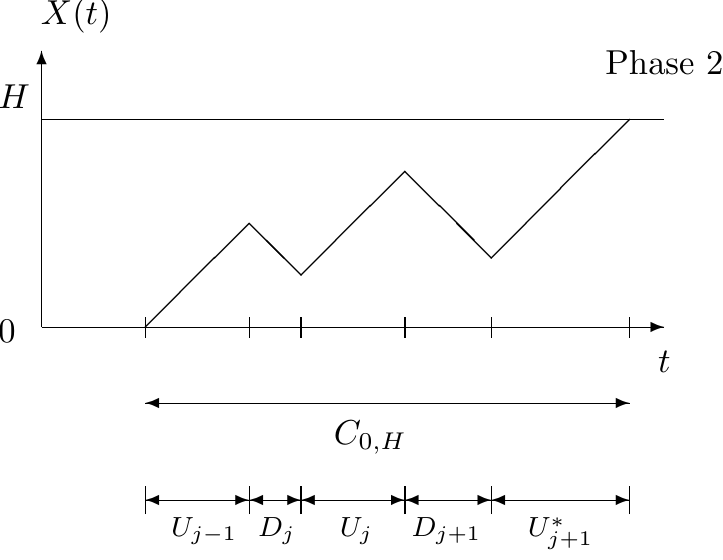}}\\
	\subfigure[]{\includegraphics[scale=0.72]{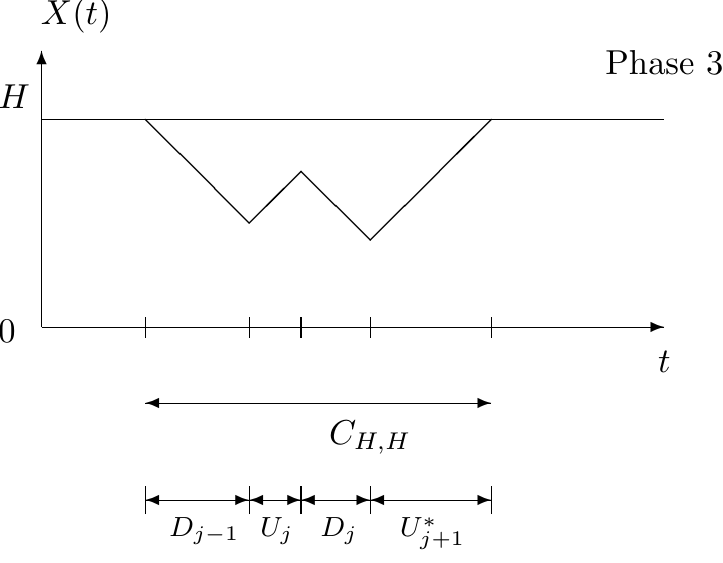}}\quad
	\subfigure[]{\includegraphics[scale=0.72]{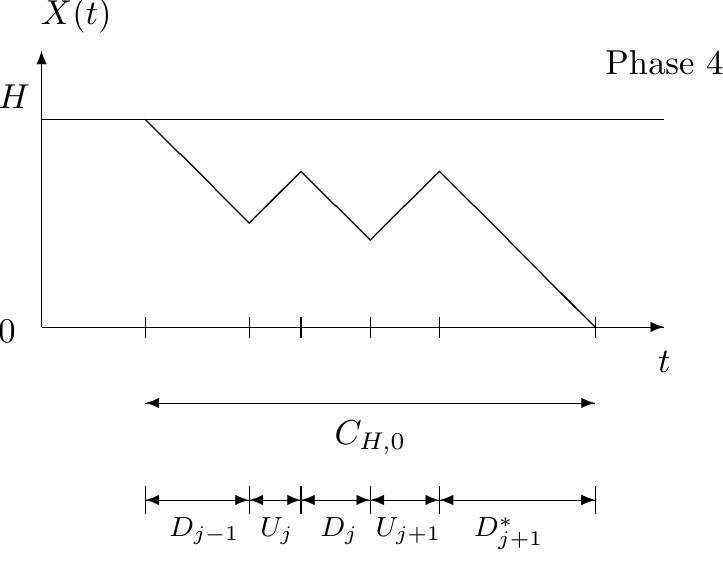}}
	\caption{Sample paths of $X(t)$ during the four different phases of motion.}
	\label{fig:2}
\end{figure}
%
\section{Distribution of the Stopping Times}\label{secStopping}
%
In this section, we assume that both the upward periods $U_j$ of the motion and the downward ones $D_j$ have exponential distribution with parameters $\lambda$ and 
$\mu$ respectively. This is the classical case concerning the asymmetric telegraph process. Moreover, in order to study the distribution of the renewal cycles, we introduce the auxiliary compound Poisson process
\begin{equation}
Y(t)=\sum_{n=0}^{N(t)}D_n, \quad t\ge 0,
\label{comp}
\end{equation}
where $D_0=0$, and  
$$
 N(t):=\max\left\{n\in\mathbb N_0: \sum_{i=1}^{n} U_i\le t\right\}
$$
is a Poisson process with intensity $\lambda$.
Hence, condition $Y(t)=s-t$, $0<t\leq s$, means that, during the interval $(0, s)$, the particle moves up for $t$ time instants and moves down for the remaining $s-t$ time instants. 
Furthermore, we remark that the random variables $\left\{D_j\right\}_{j\in \mathbb N}$ are independent from $\left\{N(t); t\ge 0\right\}$.
Hence, one has $\mathbb{P}\{Y(t)=0\}={\rm e}^{-\lambda t}$, $t\geq 0$. 
\par
Recalling that the running particle starts from the origin at time $t=0$, we define the stopping times
\begin{equation}
 T_{0,0}=\inf \left\{t>0: Y(t)\ge t\right\},
\label{T00}
\end{equation} 
\begin{equation}
 T_{0,H}=\inf \left\{t> H: Y(t) = t-H\right\},
\label{T0H}
\end{equation} 
\begin{equation}
 T_0^*=\min\left\{T_{0,0},T_{0,H}\right\}.
\label{star}
\end{equation}
Note that the stopping time $T_{0,0}$ 
can be interpreted as follows: $T_{0,0}$ is the smallest time such that 
$$
 \sum_{i=1}^{N(T_{0,0})}U_i-\sum_{i=1}^{N(T_{0,0})} D_i\leq 0,
$$
where the term on the left-hand-side gives the position of $X(t)$ after $2\,N(T_{0,0})-1$ velocity changes. Hence, 
$2\,N(T_{0,0})-1$ is the number of velocity changes such that $X(t)$ reaches the origin for the first time starting from the origin.   
Similarly, $T_{0,H}$ provides the smallest time such that 
$$
 \sum_{i=1}^{N(T_{0,H})+1}  U_i- \sum_{i=1}^{N(T_{0,H})}D_i =H,
$$
where the term on the left-hand-side gives the position of $X(t)$ after $2\,N(T_{0,H})-1$ velocity changes. 
In this case, $2\,N(T_{0,H})-1$ is the number of velocity changes such that $X(t)$ reaches the boundary $H$ for the first time starting from the origin. 
Due to definitions (\ref{T00})$\div$(\ref{star}), the following expressions for the renewal cycles related to Phase $1$ and Phase $2$
of the motion hold:
\begin{equation}
C_{0,0}={\bf 1}_{{\left\{T_0^*=T_{0,0}\right\}}}2T_{0,0},
\label{c00}
\end{equation}
\begin{equation}
C_{0,H}={\bf 1}_{{\left\{T_0^*=T_{0,H}\right\}}}\left(2T_{0,H}-H\right),
\label{c0h}
\end{equation}
where ${\bf 1}_{\bf A}$ denotes the indicator function of the set ${\bf A}$.
Indeed, one has $X(C_{0,0})=0$ if and only if $C_{0,0}=2 \sum_{i=1}^{k}U_i$ for a given $k$. Moreover, in this case we have 
$Y(T_{0,0})=\sum_{i=1}^{N(T_{0,0})}D_i=T_{0,0}$. Hence, being $N(T_{0,0})$ even, the following relation holds:
$$
 2T_{0,0}=2 \sum_{i=1}^{N(T_{0,0})}D_i=2 \sum_{i=1}^{N(T_{0,0})}U_i=C_{0,0}.
$$
Similarly, we have $X(C_{0,H})=H$ if and only if $C_{0,H}-2\sum_{i=1}^k D_i=H$ for a given $k$. 
So from the definition of $T_{0,H}$, in this case it is easy to note that $C_{0,H}=2T_{0,H}-H$. 
Figures $3$ and $4$ show some sample paths of $X(t)$ and of the corresponding compound Poisson process $Y(t)$ during Phase $1$ and Phase $2$ of the motion, respectively.
For instance,  from Figure 3(b) we have $T_{0,0} = 1.5$, which is in agreement with the fact that the total upward (downward) total time 
for $X(t)$ is 1.5, as can be seen in Figure 3(a). 
%
\begin{figure}[t]
	\centering
	\hspace*{-0.5cm}
	\subfigure[]{\includegraphics[width=5.5cm,height=4cm]{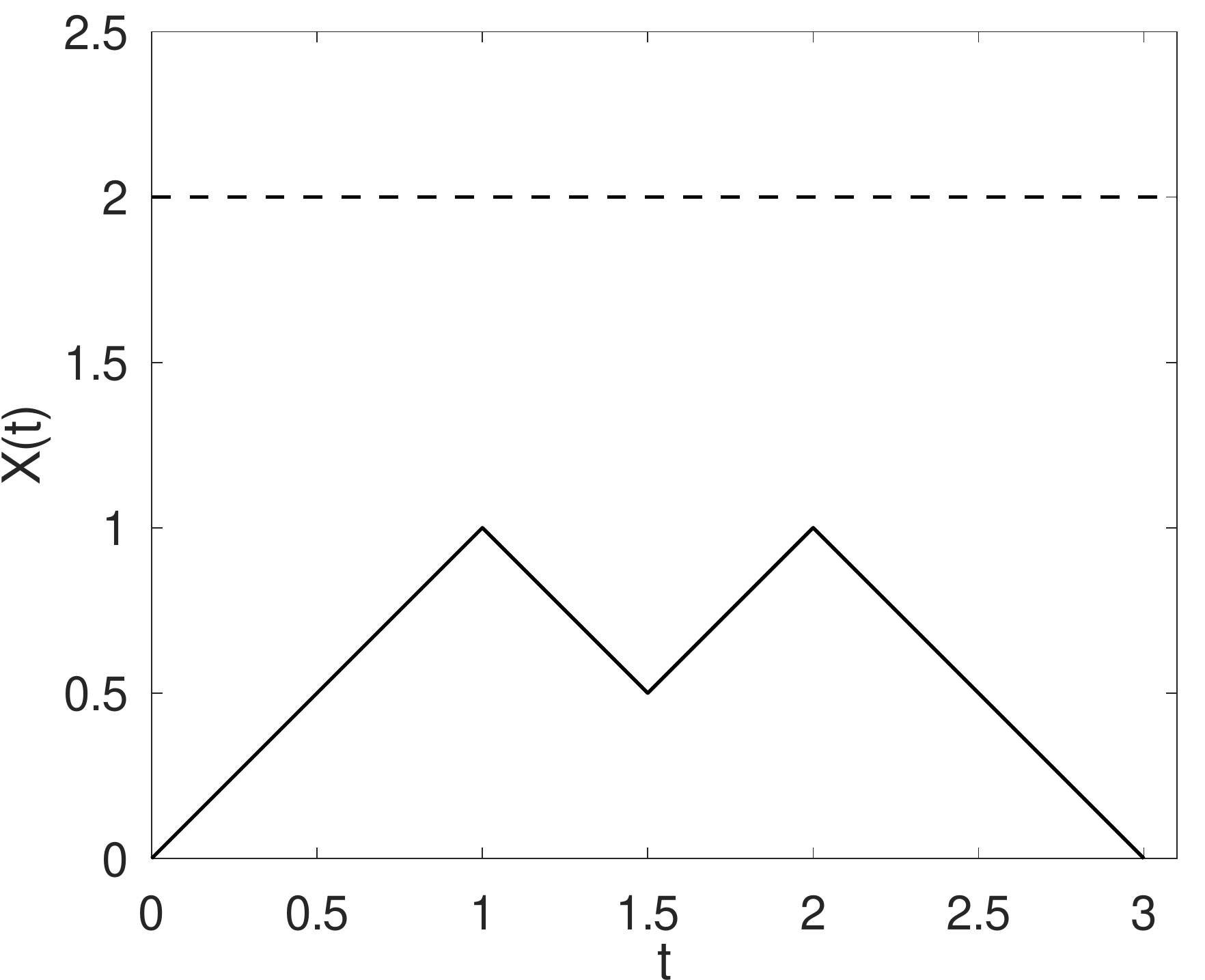}}\quad
	\subfigure[]{\includegraphics[width=5.5cm,height=3.9cm]{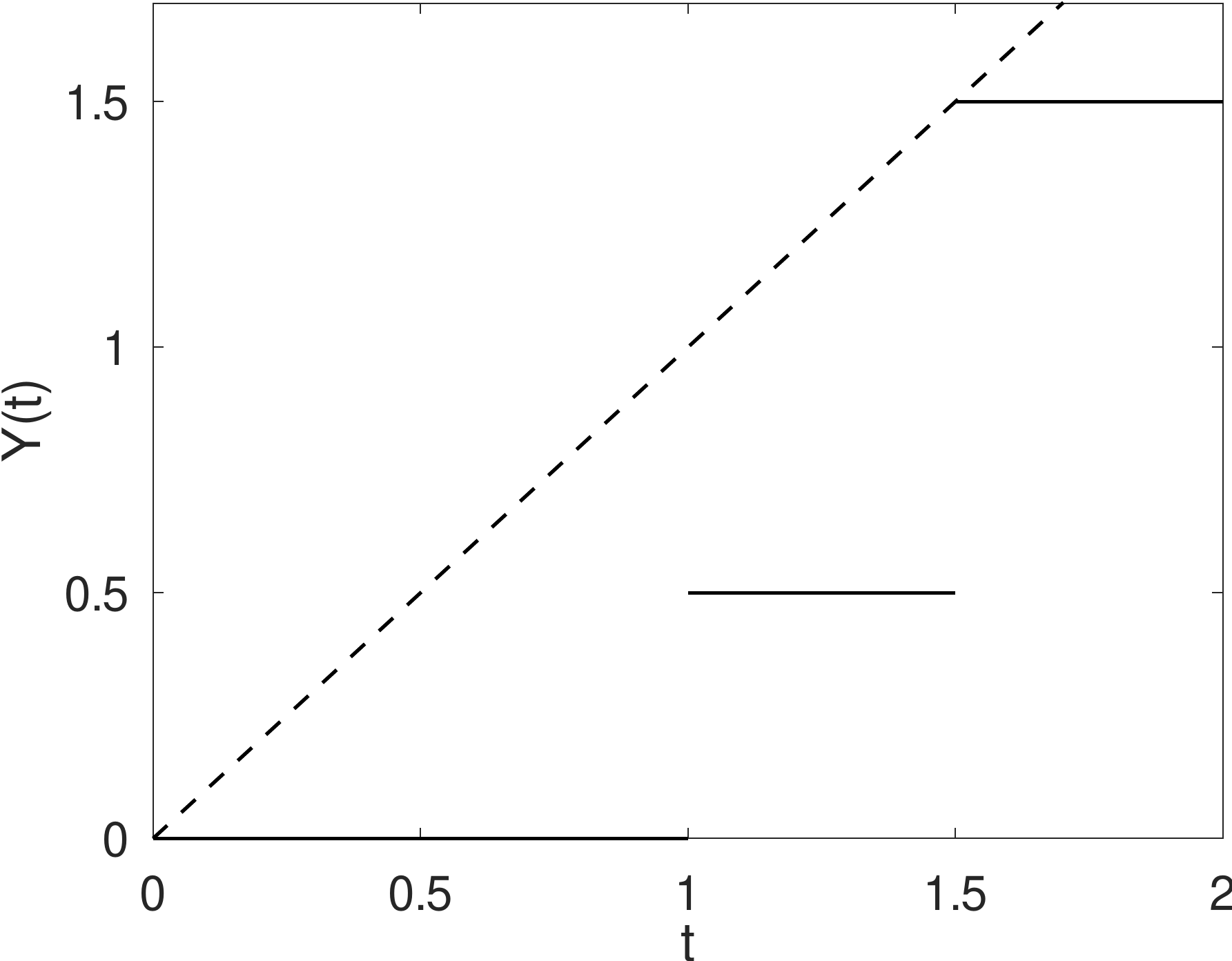}}\\
	\caption{(a) The sample path of the process $X(t)$ and (b) the corresponding compound Poisson process $Y(t)$ (solid line) during the Phase 1 of the motion, for the case $H=2$.}
	\label{fig:Figura3}
\end{figure}
%
\begin{figure}[t]
	\centering
	\hspace*{-0.4cm}
	\hspace*{-0.5cm}\subfigure[]{\includegraphics[width=5.5cm,height=4cm]{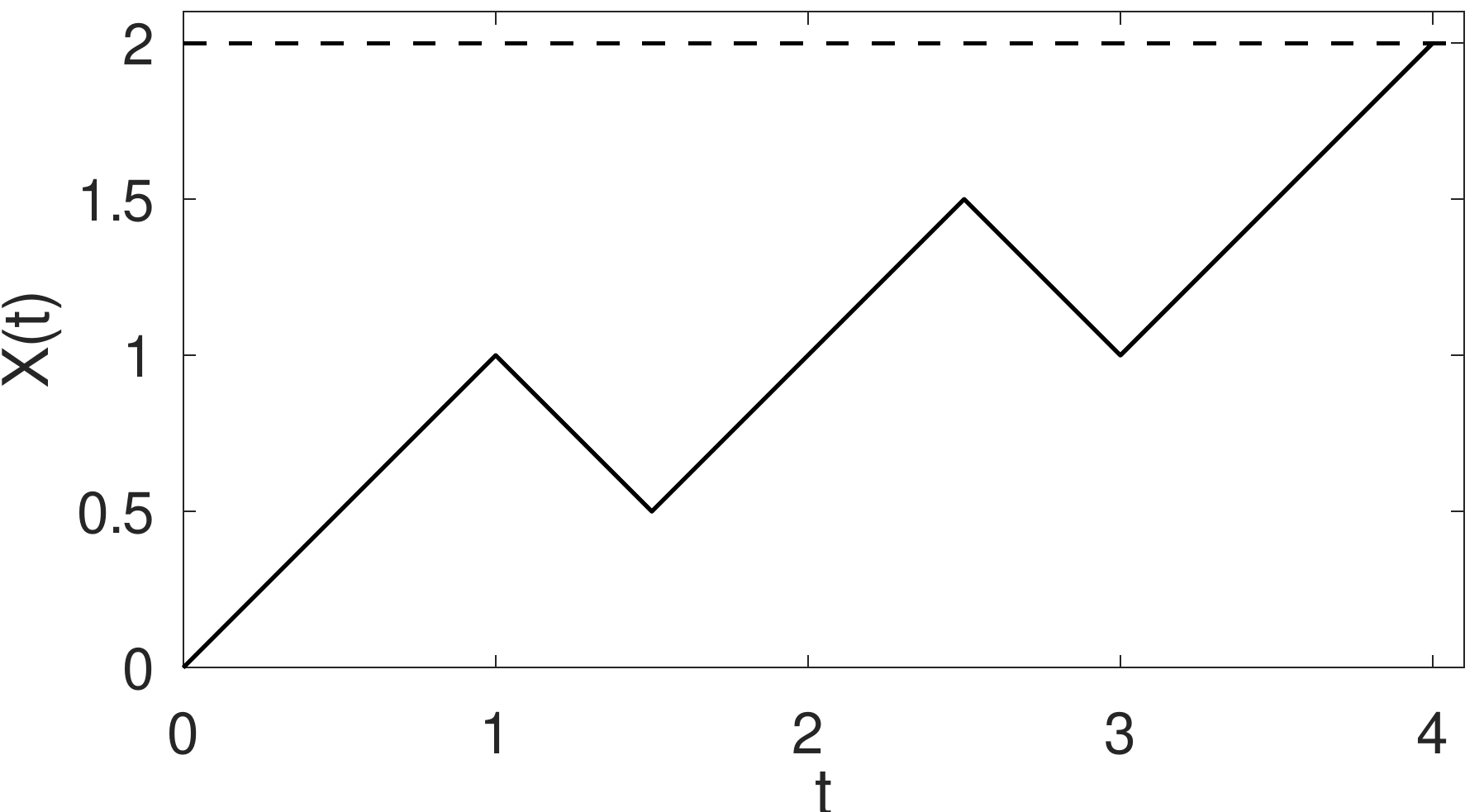}}\qquad
	\hspace*{-0.5cm} \subfigure[]{\includegraphics[width=5.5cm,height=4.1cm]{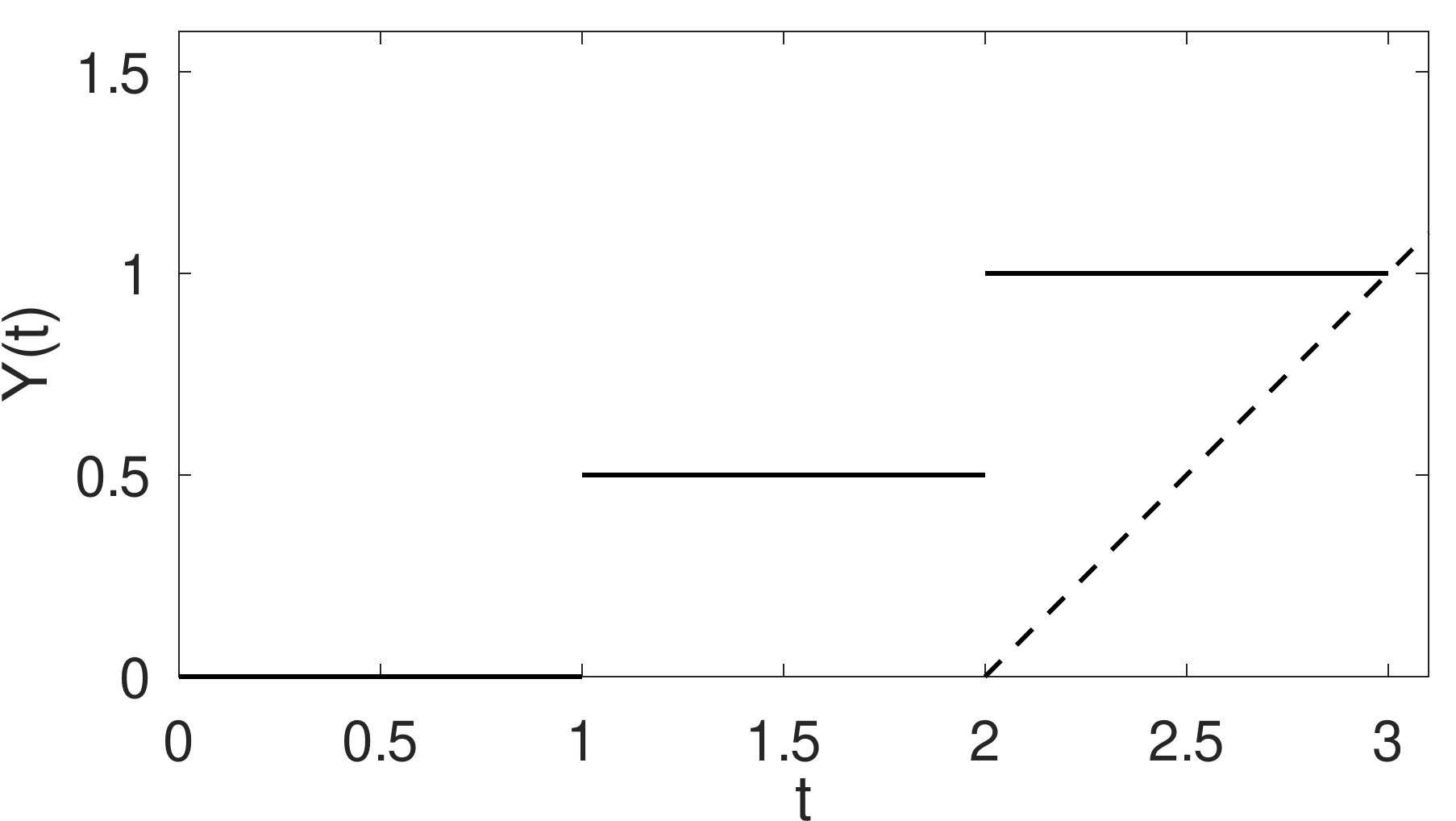}}\\
	\caption{(a) The sample path of the process $X(t)$ and (b) the corresponding compound Poisson process $Y(t)$ (solid line) during the Phase 2 of the motion, for the case $H=2$.}
	\label{fig:Figura4}
\end{figure}
\par
Let us now consider a phase in which the particle begins the motion from the level $H$.
Clearly, the first movement is downward, since the state space of the process is $\left[0, H\right]$, and it is governed by an exponentially distributed random variable, say $D$. \par
Similarly to the case of Phase $1$ and $2$, we define the following stopping times
\begin{equation}
T_{H,H}(D)=\inf\left\{t\ge 0: Y(t)= t-D\right\},
\label{THH}
\end{equation}
\begin{equation}
T_{H,0}(D)=\inf\left\{t\ge 0: Y(t)\ge t-D+H\right\},
\label{TH0}
\end{equation}
\begin{equation}
T_{H}^*(D)=\min\left\{T_{H,H}(D), T_{H,0}(D)\right\}.
\label{THstar}
\end{equation}
We remark that the stopping time $T_{H,H}(D)$ is the smallest time such that 
$$
 \sum_{i=1}^{N(T_{H,H}(D))}D_i-\sum_{i=1}^{N(T_{H,H}(D))+1} U_i+D= 0,
$$
where the term on the left-hand-side provides the position of $X(t)$ after 
$2\,N(T_{H,H}(D))-1$ velocity changes. Consequently, 
$2\,N(T_{H,H}(D))-1$ is the number of velocity changes such that $X(t)$ reaches $H$ 
for the first time starting from $H$.   
In analogy, $T_{H,0}(D)$ is the smallest time such that 
$$
 \sum_{i=1}^{N(T_{H,0}(D))}  D_i- \sum_{i=1}^{N(T_{H,0}(D))}U_i +D-H\geq 0,
$$
where the term on the left-hand-side gives the position of $X(t)$ after $2\,N(T_{H,0}(D))-1$ velocity changes. 
In this case, $2\,N(T_{H,0}(D))-1$ 
is the number of velocity changes such that $X(t)$ reaches the origin for the first time starting from the boundary $H$. 
Recalling definitions (\ref{THH})$\div$(\ref{THstar}), the renewal cycles related to Phase $3$ and Phase $4$ admit of the following expressions:
\begin{equation}
\label{defchh}
C_{H,H}(D)={\bf 1}_{\left\{T_H^*(D)=T_{H,H}(D)\right\}}2T_{H,H}(D),
\end{equation}
\begin{equation}
\label{defch0}
C_{H,0}(D)={\bf 1}_{\left\{T_H^*(D)=T_{H,0}(D)\right\}}\left(2T_{H,0}(D)+H\right).
\end{equation}
Clearly, if $D\ge H$ we have that $C_{H,H}=0$ and $C_{H,0}=H$. 
We note that $X(C_{H,H}(D))=H$ if and only if $C_{H,H}(D)=2 \sum_{i=1}^{k}D_i+2D$ for a given $k$. 
Moreover, in this case one has  
$Y(T_{H,H}(D))=\sum_{i=1}^{N(T_{H,H}(D))}D_i=T_{H,H}(D)-D$, and thus 
$2 T_{H,H}(D)=C_{H,H}(D)$. Similarly, we get 
$X(C_{H,0}(D))=0$ if and only if $C_{H,0}(D)=2 \sum_{i=1}^{k}U_i+H$ for a given $k$. 
In addition, in this case one has  
$\sum_{i=1}^{N(T_{H,0}(D))}D_i+D-H=\sum_{i=1}^{N(T_{H,0}(D))}U_i=T_{H,0}(D)$, so that
$2 T_{H,0}(D)+H=C_{H,0}(D)$. 
\par
Figures $5$ and $6$ show some sample paths of $X(t)$ and $Y(t)$ during Phases $3$ and $4$ of the motion.
%
	\begin{figure}[t]
	\centering
	\vspace*{-0.1cm}
	\hspace*{-0.4cm}
	\subfigure[]{\includegraphics[width=5.5cm,height=4.1cm]{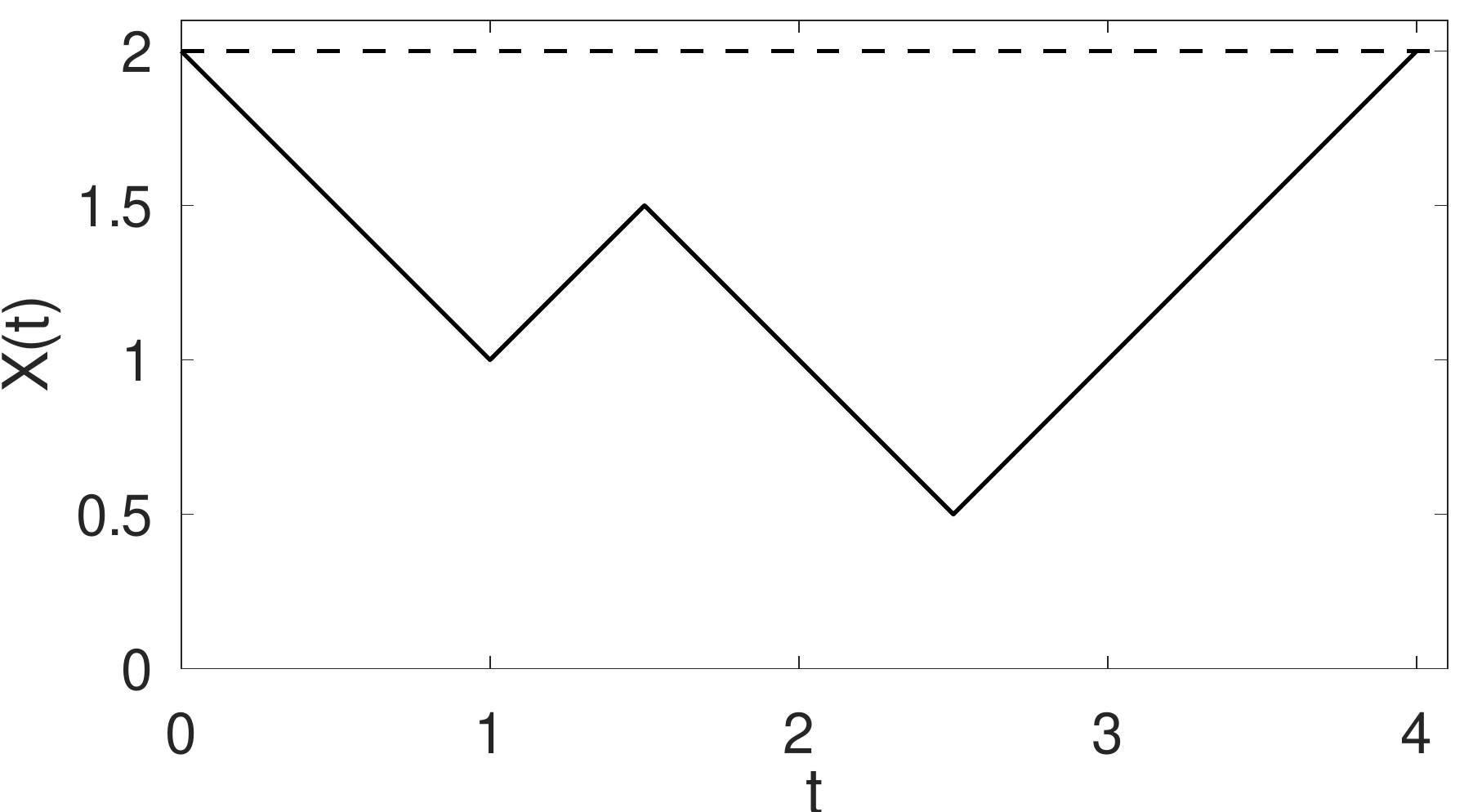}}\quad
	\subfigure[]{\includegraphics[width=5.5cm,height=4.1cm]{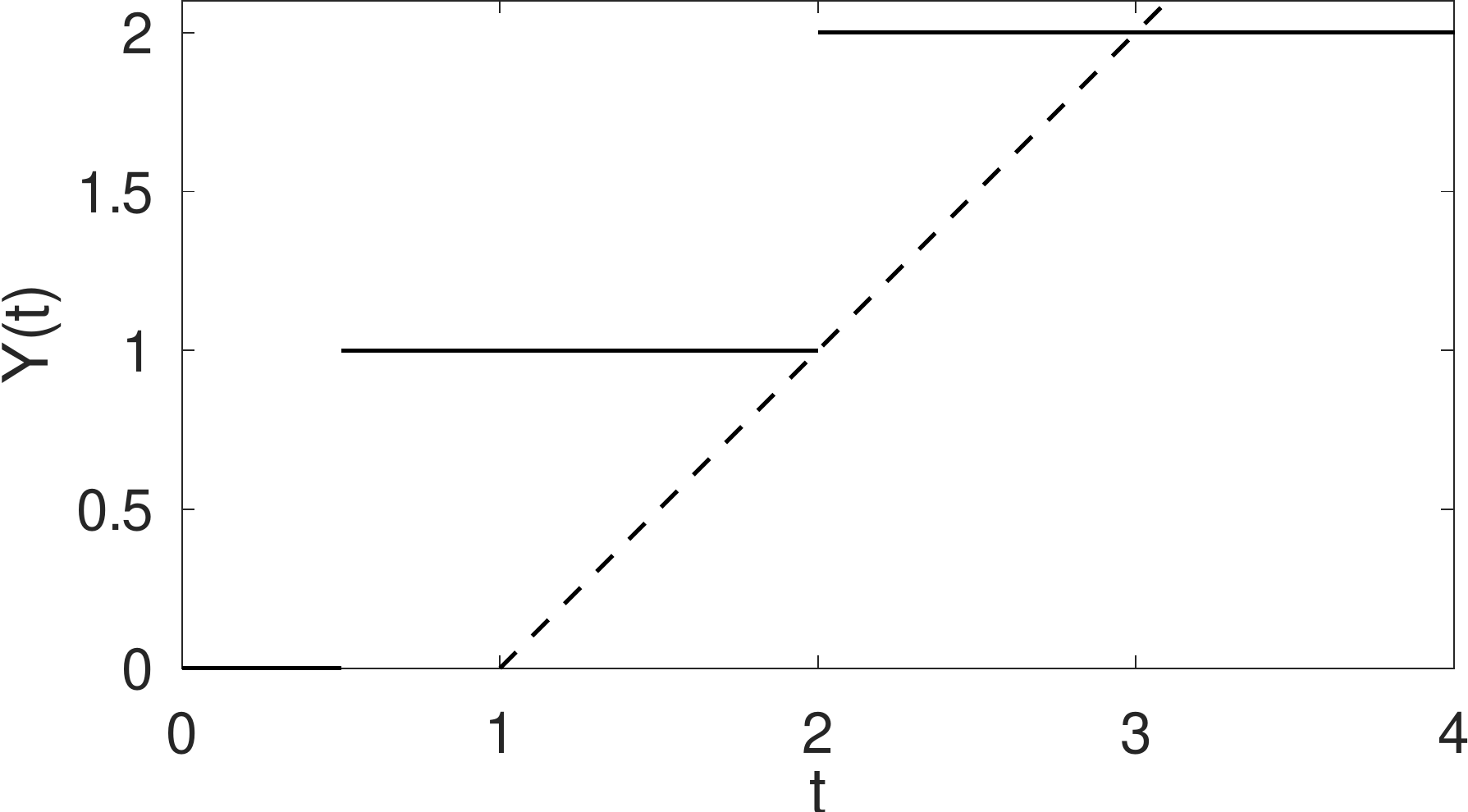}}\\
	\caption{(a) The sample path of the process $X(t)$ and (b) the corresponding compound Poisson process $Y(t)$ (solid line) during the Phase 3 of the motion, for the case $H=2$.}
	\label{fig:Figura5}
\end{figure}
%
\begin{figure}[t]
	\centering
	\hspace*{-0.5cm} \subfigure[]{\includegraphics[width=5.5cm,height=4.1cm]{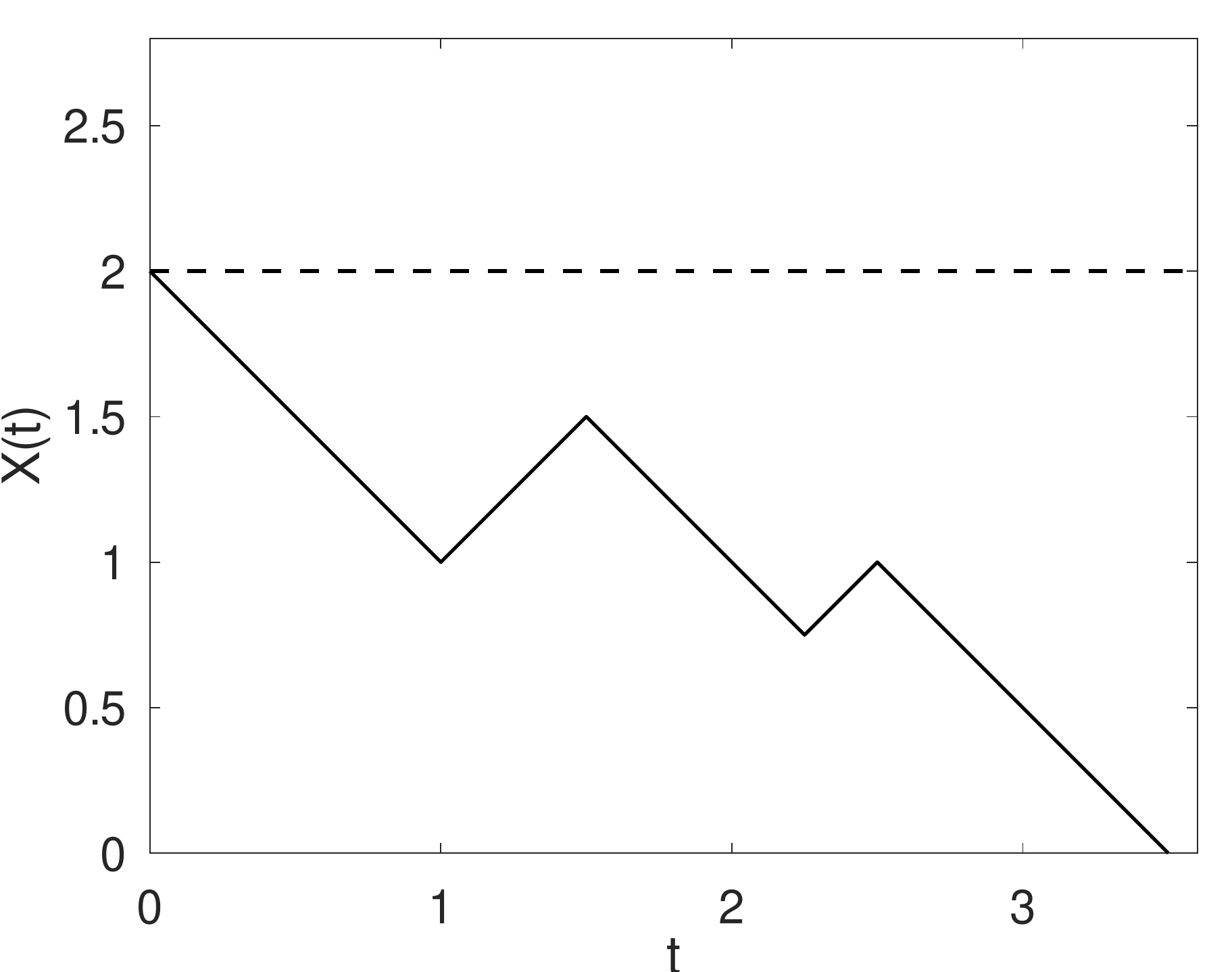}}\quad
	\hspace*{-0.5cm} \subfigure[]{\includegraphics[width=5.5cm,height=4.1cm]{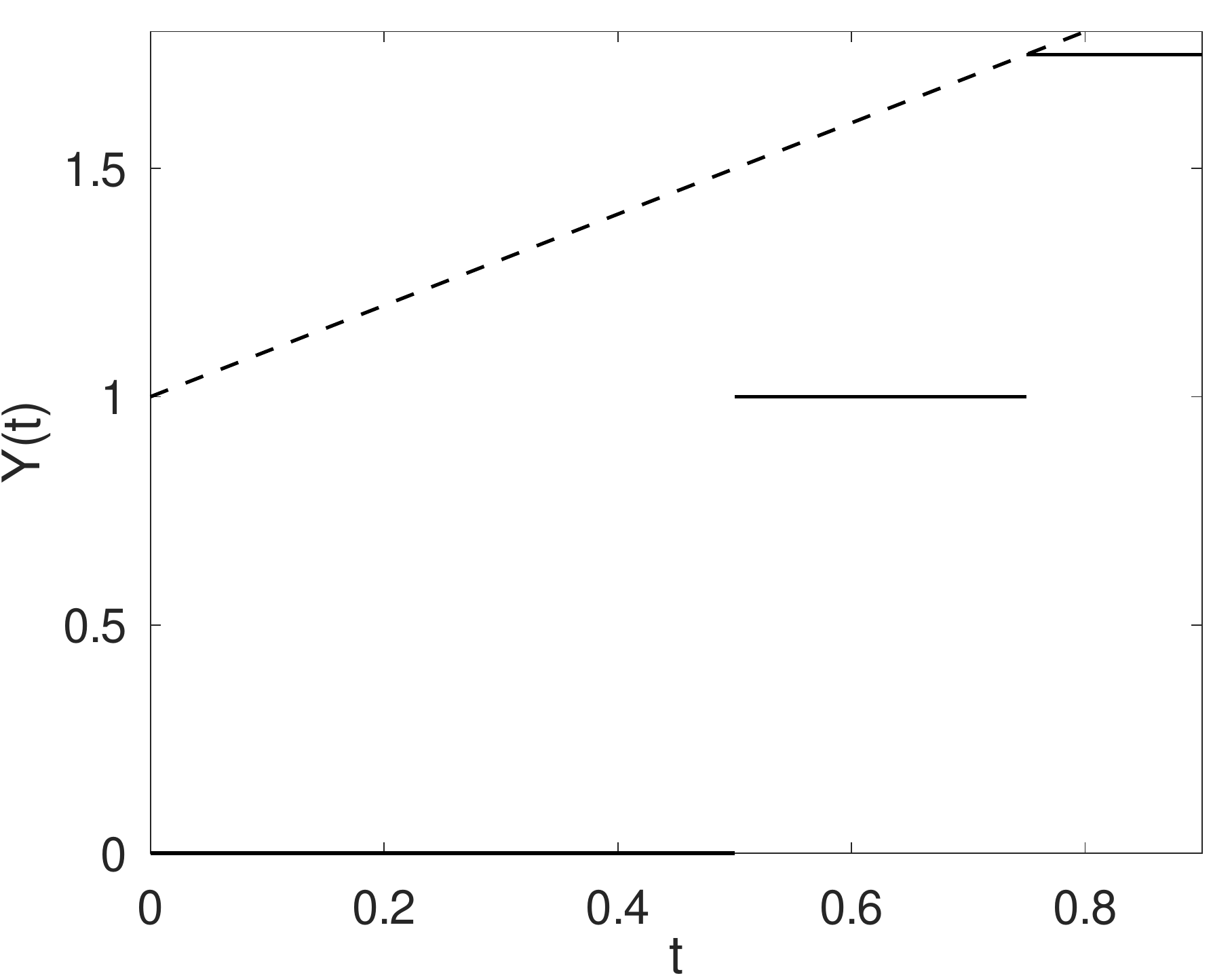}}\\
	\caption{(a) The sample path of the process $X(t)$ and (b) the corresponding compound Poisson process $Y(t)$ (solid line) during the Phase 4 of the motion, for the case $H=2$.}
	\label{fig:Figura6}
\end{figure}
\par
We point out that the above approach, based on the auxiliary compound Poisson process $Y(t)$ and its 
first-crossing time through suitable boundaries, 
has been adopted successfully in some papers concerning the telegraph process (see, for instance, \cite{Bshouty}, \cite{DCMaIuZacks2013}, \cite{DiCrescenzo}).  
%
\section{Probabilities of Phases}\label{Sec4}
%
The aim of this Section is to evaluate the probability that the particle starting from a certain boundary $u$ reaches the threshold $v$ ($u$), without hitting the other boundary 
$u$ ($v$), for $u,v\in \{0, H\}$. Hence, with reference to the stopping times defined in Eqs.\ (\ref{T00})$\div$(\ref{star}) and (\ref{THH})$\div$(\ref{THstar}), let us consider the following probabilities
\begin{equation}
P_{0,H}=\mathbb P\left\{T_0^*=T_{0,H}\right\},\qquad\qquad P_{0,0}=\mathbb P\left\{T_0^*=T_{0,0}\right\}=1-P_{0,H},
\label{4}
\end{equation}
and, for given $D$,
\begin{equation}
P_{H,0}(D)=\mathbb P\left\{T_H^*=T_{H,0}(D)\right\},\qquad
P_{H,H}(D)=\mathbb P\left\{T_H^*=T_{H,H}(D)\right\}=1-P_{H,0}(D).
\label{4bis}
\end{equation}
Due to the above definitions, $P_{0,0}$ ($P_{0,H}$) gives the probability that, starting from the origin, the particle reaches the origin (the level $H$) before 
hitting the other boundary $H$ (the level $0$). Similarly, for a given downward time $D$, $P_{H,H}(D)$ ($P_{H,0}(D)$) gives the probability that, starting from the 
level $H$, the particle reaches $H$ (the level $0$) before arriving at the other boundary $0$ (the level $H$). 
\par
In order to determine the explicit expressions of such probabilities, now we  provide the following result.
\begin{proposition}\label{prop4.1}
Let $Y(T_0^*)$ be the compound Poisson process \eqref{comp} evaluated at the stopping time $T_0^*$ introduced in 
(\ref{star}). The following Wald martingale equation holds
\begin{equation}
\mathbb E\left[e^{\theta Y(T_0^*)-\lambda T_0^*\theta/(\mu-\theta)}\right]=1, \qquad \theta<\mu.
\label{1}
\end{equation}
\end{proposition}
\begin{proof}
Let us consider the Wald martingale related to the compound Poisson process $Y(t)$, i.e.
\begin{equation}\label{2}
M(t,\theta)=\frac{\exp(\theta Y(t))}{\mathbb E\left[e^{\theta Y(t)}\right]}.
\end{equation}
Recalling that $N(t)$ is a Poisson process with intensity $\lambda$, and that the random variables $D_n$ have exponential distribution with parameter $\mu$, 
the moment generating function of the compound Poisson process $Y(t)$ is given by
\begin{equation*}
\mathbb E\left[e^{\theta Y(t)}\right]=\exp\left\{\frac{\lambda t \theta}{\mu-\theta}\right\},
\qquad \theta<\mu.
\end{equation*}
Hence, from the Eq.\ \eqref{2} one has
\begin{equation*}
M(t,\theta)=\exp\left\{\theta Y(t)-\frac{\lambda\theta t}{\mu -\theta}\right\},
\qquad \theta<\mu.
\end{equation*}
Due to the martingale property, we have, for every $t>0$, 
$$
\mathbb E[M(t,\theta)]=1,
\qquad \theta<\mu.
$$
Hence, the proof immediately follows from Doob's optional stopping Theorem.
\end{proof}
\par
Let us now determine the probabilities \eqref{4} and \eqref{4bis}.
\begin{theorem}\label{teor1}
The probabilities that a particle starting from the origin (the level $H$), hits the level $H$ (the origin), 
before hitting the origin (the level $H$) again, are given, for $\lambda\neq \mu$, 
by
\begin{equation}
P_{0,H}=\frac{\mu -\lambda}{\mu-\lambda \, e^{(\lambda-\mu)H}},
\label{probP0H}
\end{equation} 
and
\begin{equation}
P_{H,0}=\frac{\lambda -\mu}{\lambda-\mu \, e^{(\mu-\lambda)H}}.
\label{probPH0}
\end{equation}
Similarly, the probabilities that a particle starting from the origin (the level $H$), hits the origin (the level $H$), 
before hitting the level $H$ (the origin), for $\lambda\neq \mu$ are expressed as
\begin{equation}
P_{0,0}=1-P_{0,H}=\frac{\lambda e^{H(\lambda -\mu)}-\lambda}{\lambda e^{H(\lambda-\mu)}-\mu},
\label{probP00}
\end{equation} 
and
\begin{equation}
P_{H,H}=1-P_{H,0}=\frac{\mu \left(1- e^{-(\lambda - \mu) H}\right)}{\lambda-\mu e^{-(\lambda-\mu) H}}.
\label{probPHH}
\end{equation}
\end{theorem}
\begin{proof}
For a given $i\in \mathbb{N}$, let $D_i\sim Exp(\mu)$ be the random duration of the downward period during which $X(t)$ crosses 
the origin for the first time. Hence, the following equalities in distribution hold:
\begin{equation}\label{5}
\begin{aligned} 
& Y(T_0^*){\bf 1}_{\left\{T_0^*=T_{0,0}\right\}}\stackrel{d}{=}T_{0,0}+D_i-D_i^*,\\
& Y(T_0^*){\bf 1}_{\left\{T_0^*=T_{0,H}\right\}}\stackrel{d}{=}T_{0,H}-H.
\end{aligned}
\end{equation}
Clearly, for the memoryless property, we have $D_i-D_i^*\sim Exp(\mu)$, with $D_i-D_i^*$ independent from $T_{0,0}$. 
Recalling Eq.\ \eqref{star} and thanks to \eqref{5}, Eq.\ \eqref{1} becomes, 
for all $\theta<\mu$,
\begin{equation}\label{7}
\frac{\mu}{\mu -\theta }\mathbb E\left[e^{\omega T_{0,0}}{\bf 1}_{\left\{T_0^*=T_{0,0}\right\}}\right]+e^{-H\theta}\mathbb E\left[e^{\omega T_{0,H}}{\bf 1}_{\left\{T_0^*=T_{0,H}\right\}}\right]=1,
\end{equation} 
where 
\begin{equation}\label{6}
 \omega:= 
 \frac{\theta\cdot(\mu -\lambda-\theta)}{\mu-\theta}. 
\end{equation}
From Eq.\ (\ref{6}) we have $\theta^2+\theta(\lambda-\mu-\omega)-\mu\omega=0$. 
Denoting by $\theta_1(\omega)\leq \theta_2(\omega)$ the roots 
of the latter equation, and recalling that 
$\theta<\mu$, cf.\ Proposition \ref{prop4.1}, we have
\begin{equation*}
\theta_{1}(\omega)\leq \theta_{2}(\omega)< \mu,
\end{equation*}
which is equivalent to
$$
 \omega\leq  \left(\sqrt{\lambda}-\sqrt{\mu}\right)^2.
$$ 
Let us now define
		\begin{equation*}
		F_{0,0}(\omega)=\mathbb E\left[e^{\omega T_{0,0}}{\bf 1}_{\left\{T_0^*=T_{0,0}\right\}}\right],
		\end{equation*}
		and
		\begin{equation*}
		F_{0,H}(\omega)=\mathbb E\left[e^{\omega T_{0,H}}{\bf 1}_{\left\{T_0^*=T_{0,H}\right\}}\right].
		\end{equation*}
		Substituting the roots $\theta_{1,2}(\omega)$ in Eq.\ \eqref{7}, we obtain the system
		\begin{equation*}
		\left\{\begin{matrix}
		&\frac{\mu}{\mu -\theta_1 (\omega)}F_{0,0}(\omega)+e^{-H\theta_1(\omega)}F_{0,H}(\omega)=1,\\
		&\frac{\mu}{\mu -\theta_2 (\omega)}F_{0,0}(\omega)+e^{-H\theta_2(\omega)}F_{0,H}(\omega)=1,
		\end{matrix}\right.
		\end{equation*}
		whose solutions are given by
		\begin{equation}\label{10}
		\begin{aligned}
		&F_{0,0}(\omega)=\frac{\left(e^{-H\theta_1(\omega)}-e^{-H\theta_2(\omega)}\right)\left(\mu-\theta_1(\omega)\right)\left(\mu-\theta_2(\omega)\right)}{\mu\left[\left(\mu-\theta_1(\omega)\right)e^{-H\theta_1(\omega)}-\left(\mu-\theta_2(\omega)\right)e^{-H\theta_2(\omega)}\right]},\\
		&F_{0,H}(\omega)=\frac{\theta_2(\omega)-\theta_1(\omega)}{\left(\mu-\theta_1(\omega)\right)e^{-H\theta_1(\omega)}-\left(\mu-\theta_2(\omega)\right)e^{-H\theta_2(\omega)}}.
		\end{aligned}
		\end{equation}
		Finally, by taking $\theta=0$ we have $\omega=0$, so that from 
		the second of Eqs.\ \eqref{10} one has
		\begin{equation*}
		P_{0,H}=F_{0,H}(0)=\frac{\lambda-\mu}{\lambda e^{H(\lambda-\mu)}-\mu}.
		\end{equation*}
		Hence, since $P_{0,0}=1-P_{0,H}$, we immediately obtain Eq.\ (\ref{probP00}).
		In the cases of Phase 3 and Phase 4, for a given $D$, the martingale equation reads
		\begin{equation}\label{11}
		e^{-\theta D}{F}_{H,H}(\omega|D)+e^{\theta (H-D)}\left(\frac{\mu}{\mu-\theta }\right){F}_{H,0}(\omega|D)=1,
		\qquad 
		\omega\leq  \left(\sqrt{\lambda}-\sqrt{\mu}\right)^2,
		\end{equation}
		where
		\begin{equation*}
		\begin{aligned}
		&{F}_{H,H}(\omega|D)=\mathbb E\left[e^{\omega T_{H,H}(D)}{\bf 1}_{\left\{T_H^*(D)=T_{H,H}(D)\right\}}\right],\\
		&{F}_{H,0}(\omega|D)=\mathbb E\left[e^{\omega T_{H,0}(D)}{\bf 1}_{\left\{T_H^*(D)=T_{H,0}(D)\right\}}\right].
		\end{aligned}
		\end{equation*}
		Following a similar reasoning, the solution of the system obtained substituting $\theta_1(\omega)$ and $\theta_2(\omega)$ in Eq.\ \eqref{11} 
		is given by
		\begin{equation}
		\begin{aligned}
		&{F}_{H,H}(\omega|D)=\frac{\left(\mu-\theta_2(\omega)\right)e^{\theta_1(\omega)(H-D)}-\left(\mu-\theta_1(\omega)\right)e^{\theta_2(\omega)(H-D)}}{\left(\mu-\theta_2(\omega)\right)e^{-\theta_1(\omega)(D-H)-\theta_2(\omega)D}-\left(\mu-\theta_1(\omega)\right)e^{-\theta_1(\omega)D-\theta_2(\omega)(D-H)}},\\		
		&{F}_{H,0}(\omega|D)=\frac{e^{-\theta_1(\omega)D}-e^{-\theta_2(\omega)D}}{\frac{\mu}{\mu-\theta_2(\omega)}e^{-\theta_1(\omega)D+\theta_2(\omega)(H-D)}-\frac{\mu}{\mu-\theta_1(\omega)}e^{-\theta_2(\omega)D+\theta_1(\omega)(H-D)}}.	
		\end{aligned}
		\label{10bis}
		\end{equation}
Hence, for $\theta=0$ one has  $\omega=0$, and thus 
\begin{equation}
P_{H,0}(D)=F_{H,0}(0|D)=\frac{\lambda\left(e^{(\mu-\lambda)D}-1\right)}{\mu e^{(\mu-\lambda)H}-\lambda},
\label{probPH0D}
\end{equation}
and thus
\begin{equation*}
P_{H,H}(D)=1-P_{H,0}(D)= \frac{\mu e^{(\mu-\lambda)H}-\lambda e^{(\mu-\lambda)D}}{\mu e^{(\mu-\lambda)H}-\lambda}.
\end{equation*}
Being $D$ an exponentially distributed random variable with parameter $\mu$, and recalling Eq.\ (\ref{probPH0D}), finally the probability $P_{H,0}$ can be 
evaluated as follows
\begin{equation*}
P_{H,0}=\mu\int_{0}^{H}e^{-\mu x}P_{H,0}(x) \,{\rm d}x
+\mu\int_H^{+\infty} e^{-\mu x}  \,{\rm d}x
=\frac{\mu-\lambda}{\mu e^{(\mu-\lambda)H}-\lambda}.
\end{equation*} 
Moreover, noting that $P_{H,H}=1-P_{H,0}$, we obtain Eq.\ (\ref{probPHH}).
This completes the proof. Note that, by symmetry, the probabilities $P_{H,0}$ and $P_{H,H}$ can be obtained also by interchanging $\lambda$ with $\mu$ 
in the expressions of $P_{0,0}$ and $P_{0,H}$, respectively.
\end{proof}
\par
Figures \ref{fig:Figura7}--\ref{fig:Figura7bis} provide the probabilities $P_{0,0}$ and $P_{H,H}$ for various choices of the parameters. They show that $P_{0,0}$ is increasing in $\lambda$ and $H$, and is decreasing in $\mu$. 
Similarly, $P_{H,H}$ is increasing in $\mu$ and $H$, and is decreasing in $\lambda$. 
\par
Thanks to Theorem \ref{teor1}, we immediately get the following result. 
\begin{corollary}\label{prop1}
For $\lambda=\mu$, the probabilities (\ref{probP0H})$\div$(\ref{probPHH}) are given by
\begin{equation*}
P_{0,0}=P_{H,H}=\frac{\mu H}{1+\mu H},
\end{equation*} 
and
\begin{equation*}
P_{0,H}=P_{H,0}=\frac{1}{1+\mu H}.
\end{equation*}
\end{corollary}
\par
%
	\begin{figure}[t]
	\centering
	\hspace*{-0.5cm}
	\subfigure[]{\includegraphics[width=5.5cm,height=4cm]{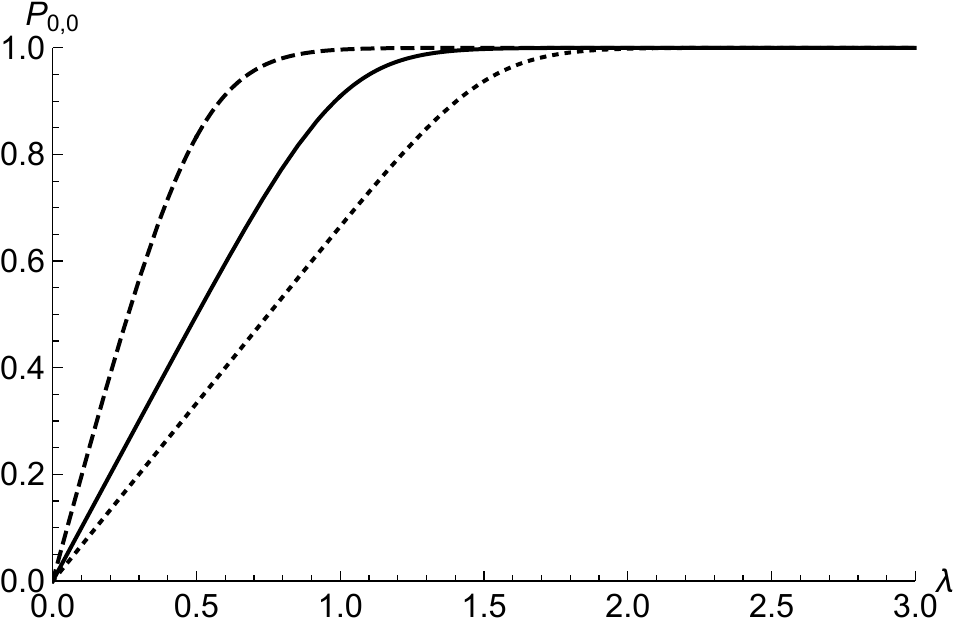}}\quad
	\subfigure[]{\includegraphics[width=6cm,height=4cm]{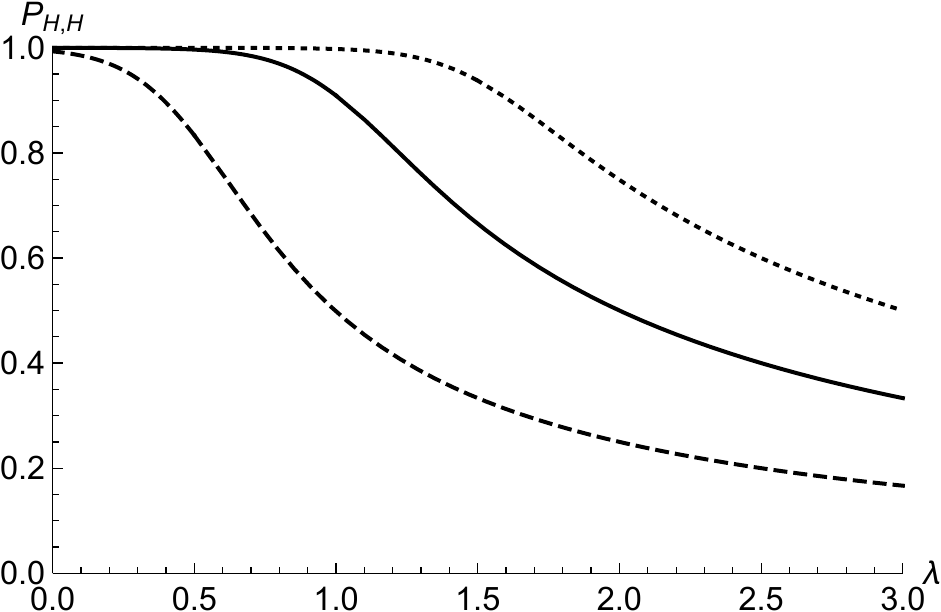}}\\
	\caption{Plot (a) (Plot (b)) shows $P_{0,0}$ ($P_{H,H}$) as a function of $\lambda$, for $\mu=0.5$ (dashed line), $\mu=1$ (solid line) and $\mu=1.5$ (dotted line). 
	In both cases it is $H=10$.}
	\label{fig:Figura7}
\end{figure}
%
\begin{figure}[t]
	\centering
	\hspace*{-0.5cm}
	\subfigure[]{\includegraphics[width=5.5cm,height=4cm]{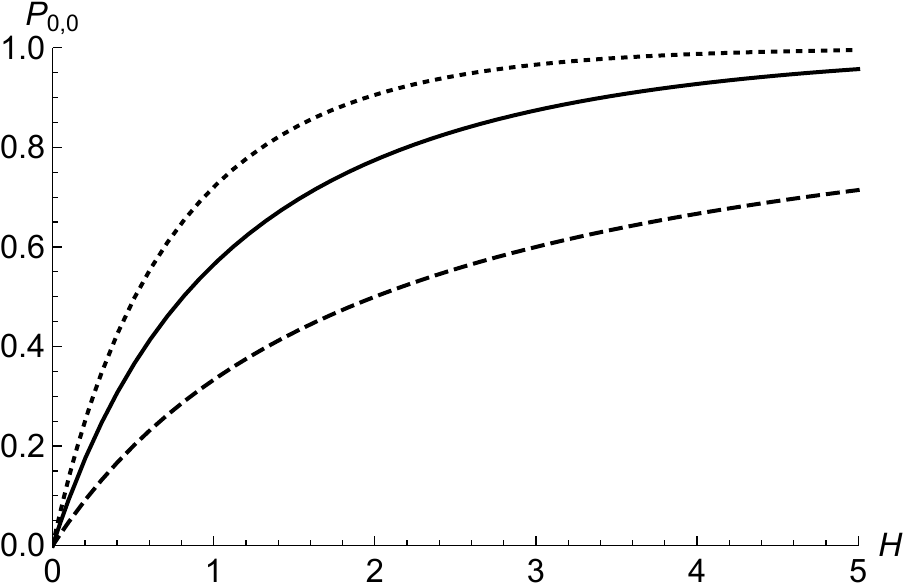}}\quad
	\subfigure[]{\includegraphics[width=6cm,height=4cm]{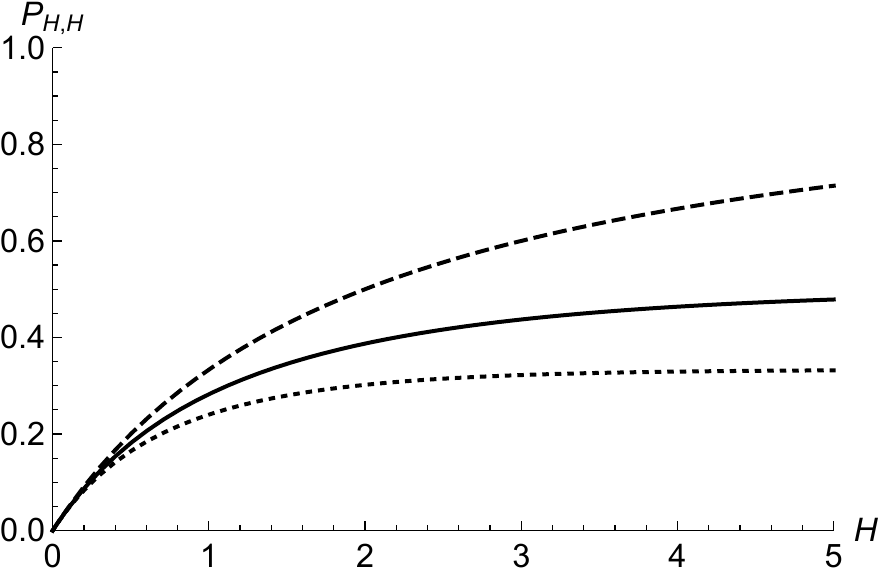}}\\
	\caption{Plot (a) (Plot (b)) shows $P_{0,0}$ ($P_{H,H}$) as a function of $H$, for $(\lambda,\mu)=(0.5,0.5)$ (dashed line), $(\lambda,\mu)=(1,0.5)$ (solid line) 
	and $(\lambda,\mu)=(1.5,0.5)$ (dotted line).}
	\label{fig:Figura7bis}
\end{figure}
%
\section{Expected Values of Stopping Times}\label{Sec5}
%
In this section we exploit the martingale equation given in Section \ref{Sec4} to obtain the explicit expression of the 
mean value of the first-passage times defined in (\ref{T00})$\div$(\ref{star}) and  (\ref{THH})$\div$(\ref{THstar}). 
We now focus on the stopping times $T_{u,v}$, for $u,v\in\{0,H\}$. Clearly, these are not honest random variables, whereas  
$T_{u,v}{\bf 1}_{\left\{T_u^*=T_{u,v}\right\}}$ are, for $u,v\in\{0,H\}$. 
With reference to Eq.\ (\ref{c00}) (Eq.\ (\ref{defchh})), we note that 
$T_{0,0}{\bf 1}_{\left\{T_0^*=T_{0,0}\right\}}$ ($T_{H,H}{\bf 1}_{\left\{T_H^*=T_{H,H}\right\}}$)
represents the half of the time that the process $X(t)$ spends to return to the origin (the boundary $H$) 
for the first time without hitting the boundary $H$ (the origin). 
The meaning of $T_{0,H}{\bf 1}_{\left\{T_0^*=T_{0,H}\right\}}$ and $T_{H,0}{\bf 1}_{\left\{T_H^*=T_{H,0}\right\}}$ 
is similar, due to Eqs.\  (\ref{c0h}) and (\ref{defch0}).  
\begin{proposition}\label{prop2}
	The expected values of $T_{0,H}{\bf 1}_{\left\{T_0^*=T_{0,H}\right\}}$ and $T_{0,0}{\bf 1}_{\left\{T_0^*=T_{0,0}\right\}}$ are given, for $\lambda\neq \mu$, by
	\begin{eqnarray*}
	&& \hspace*{-0.8cm}
	\mathbb E\left[T_{0,H}{\bf 1}_{\left\{T_0^*=T_{0,H}\right\}}\right]= 
	\frac{{\rm e}^{(\mu-\lambda)H}\left\{ 2\lambda \mu [{\rm e}^{(\mu-\lambda)H}-1] +H (\lambda-\mu) [\lambda^2+\mu^2 {\rm e}^{(\mu-\lambda)H}] \right\}}{(\lambda-\mu)(\lambda-\mu {\rm e}^{(\mu-\lambda)H})^2},
	\nonumber
	\\
	&& \hspace*{-0.8cm}
       \\
	&& \hspace*{-0.6cm}
    \mathbb E\left[T_{0,0}{\bf 1}_{\left\{T_0^*=T_{0,0}\right\}}\right]= \frac{\lambda\left\{\lambda-\mu e^{2H(\mu-\lambda)}-e^{H(\mu-\lambda)}(\lambda-\mu)\left[1+H(\lambda+\mu)\right]\right\}}{(\lambda-\mu)\left(\lambda-\mu e^{H(\mu-\lambda)}\right)^2}.
    \nonumber
	\\
	&& \hspace*{-0.8cm}
	\end{eqnarray*}
	Moreover, for the expected values of $T_{H,H}{\bf 1}_{\left\{T_H^*=T_{H,H}\right\}}$ and $T_{H,0}{\bf 1}_{\left\{T_H^*=T_{H,0}\right\}}$ we have, for $\lambda\neq \mu$,
	\begin{equation*}
	\mathbb E\left[T_{H,H}{\bf 1}_{\left\{T_H^*=T_{H,H}\right\}}\right]=\frac{\mu\left\{\lambda-\mu e^{2H(\mu-\lambda)}+(\mu-\lambda)e^{(\mu-\lambda)H}\left[1+H(\mu+\lambda)\right]\right\}}{(\lambda-\mu)\left(\lambda-\mu e^{(\mu-\lambda)H}\right)^2}, 
	\end{equation*} 
	and
	\begin{equation*}
	\mathbb E\left[T_{H,0}{\bf 1}_{\left\{T_H^*=T_{H,0}\right\}}\right]=\frac{\lambda\mu\left\{2+H(\mu-\lambda)+e^{H(\mu-\lambda)}\left[-2+H(\mu-\lambda)\right]\right\}}{(\mu-\lambda)\left(\lambda-\mu e^{(\mu-\lambda)H}\right)^2}.
	\end{equation*} 
\end{proposition}
\begin{proof}
Recalling Eq.\ (\ref{10}), we have
\begin{equation*}
\begin{aligned}
\mathbb 
E\left[T_{0,H}{\bf 1}_{\left\{T_0^*=T_{0,H}\right\}}\right]&=\left.\frac{\partial}{\partial \omega}F_{0,H}(\omega)\right|_{\omega=0}\\
&=\frac{\lambda +\mu}{(\lambda-\mu)\left(\mu-\lambda e^{H(\lambda-\mu)}\right)}
+\frac{\mu+\mu^2H +\lambda e^{H(\lambda-\mu)}(1+\lambda H)}{\left(\mu-\lambda e^{H(\lambda -\mu)}\right)^2}.
\end{aligned}
\end{equation*}
Similarly,
\begin{equation*}
\begin{aligned}
\mathbb E\left[T_{0,0}{\bf 1}_{\left\{T_0^*=T_{0,0}\right\}}\right]&=
&\frac{\lambda\left[\lambda-\mu e^{2H(\mu-\lambda)}-e^{H(\mu-\lambda)}(\lambda-\mu)\left(1+H(\lambda+\mu)\right)\right]}{(\lambda-\mu)\left(\lambda-\mu e^{H(\mu-\lambda)}\right)^2}.
\end{aligned}
\end{equation*}
By setting 
\begin{equation*}
M_{H,H}(\lambda,\mu, H, D)=\mathbb E\left[T_{H,H}(D){\bf 1}_{\left\{T_H^*(D)=T_{H,H}(D)\right\}}\right]
\end{equation*}
and
\begin{equation*}
M_{H,0}(\lambda,\mu, H, D)=\mathbb E\left[T_{H,0}(D){\bf 1}_{\left\{T_H^*(D)=T_{H,0}(D)\right\}}\right],
\end{equation*}
from Eq.\ (\ref{10bis}) we have 
\begin{eqnarray*}
&& \hspace*{-1.8cm}
M_{H,H}(\lambda,\mu, H, D)=\left\{D\left(\lambda-\mu e^{(\mu-\lambda)H}\right)\left(\lambda^2e^{(\mu-\lambda)D}+\mu^2 e^{(\mu-\lambda)H}\right)\right.\\
&& \hspace*{0.5cm}
+ \left.\lambda\mu e^{(\mu-\lambda)H}\left(e^{(\mu-\lambda)D}-1\right)\left(2+H(\lambda+\mu)\right)\right\}\frac{1}{(\lambda-\mu)\left(\lambda-\mu e^{(\mu-\lambda) H}\right)^2},
\end{eqnarray*}
and
\begin{eqnarray*}
&& \hspace*{-1.4cm}
M_{H,0}(\lambda,\mu, H, D)=\lambda\left\{D\left(\mu+\lambda e^{(\mu-\lambda)D}\right)\left(\mu-\mu e^{(\mu-\lambda)H}\right) \right.\\
&& \hspace*{0.5cm}
+\left.\left(e^{(\mu-\lambda)D}-1\right)\left(\lambda+\lambda\mu H +e^{(\mu-\lambda)H}(1+\lambda H)\mu\right)\right\}
\frac{1}{(\lambda-\mu)\left(\lambda-\mu e^{(\mu-\lambda) H}\right)^2}.
\end{eqnarray*}
Hence, due to Eqs.\ (\ref{defchh}) and (\ref{defch0}), and recalling that, for $D\geq H$ it is $C_{H,H}=0$ and $C_{H,0}=H$, we finally obtain
\begin{eqnarray*}
&& \hspace*{-2.5cm}
\mathbb E\left[T_{H,H}{\bf 1}_{\left\{T_H^*=T_{H,H}\right\}}\right]
=\mu\int_0^H e^{-\mu x}M_{H,H}(\lambda, \mu, H, x) \,{\rm d}x\\
&& \hspace*{0.8cm}
=\frac{\mu\left\{\lambda-\mu e^{2H(\mu-\lambda)}+(\mu-\lambda)e^{(\mu-\lambda)H}\left[1+H(\mu+\lambda)\right]\right\}}{(\lambda-\mu)\left(\lambda-\mu e^{(\mu-\lambda)H}\right)^2}, 
\end{eqnarray*}
and
\begin{eqnarray*}
&& \hspace*{-2.5cm}
\mathbb E\left[T_{H,0}{\bf 1}_{\left\{T_H^*=T_{H,0}\right\}}\right]
=\mu\int_0^H e^{-\mu x}M_{H,0}(\lambda, \mu, H, x) \,{\rm d}x\\ 
&& \hspace*{0.7cm}
=\frac{\lambda\mu\left\{2+H(\mu-\lambda)+e^{H(\mu-\lambda)}\left[-2+H(\mu-\lambda)\right]\right\}}{(\mu-\lambda)\left(\lambda-\mu e^{(\mu-\lambda)H}\right)^2}.
\end{eqnarray*}
The proof is thus completed.
\end{proof}
\par
From Proposition \ref{prop2} one has the following 
\begin{corollary}
In the case $\lambda=\mu$, we have
\begin{eqnarray*}
	&& \hspace*{-0.8cm}
	\mathbb E\left[T_{0,H}{\bf 1}_{\left\{T_0^*=T_{0,H}\right\}}\right]= \frac{H(6+6\mu H+H^2 \mu^2)}{6(1+\mu H)^2},
	\\
	&& \hspace*{-0.8cm}
    \mathbb E\left[T_{0,0}{\bf 1}_{\left\{T_0^*=T_{0,0}\right\}}\right]=\mathbb E\left[T_{H,H}{\bf 1}_{\left\{T_H^*=T_{H,H}\right\}}\right]= \frac{\mu H^2 (3+2 \mu H)}{6 (1+\mu H)^2},
    \\
	&& \hspace*{-0.8cm}
\mathbb E\left[T_{H,0}{\bf 1}_{\left\{T_H^*=T_{H,0}\right\}}\right]=\frac{\mu^2 H^3 }{6(1+\mu H)^2}.
	\end{eqnarray*}
\end{corollary}
\par
The expected values of the renewal cycles are provided hereafter.
%
\begin{figure}[t]
\vspace*{-0.5cm}
\centering
{\includegraphics[width=10cm,height=6cm]{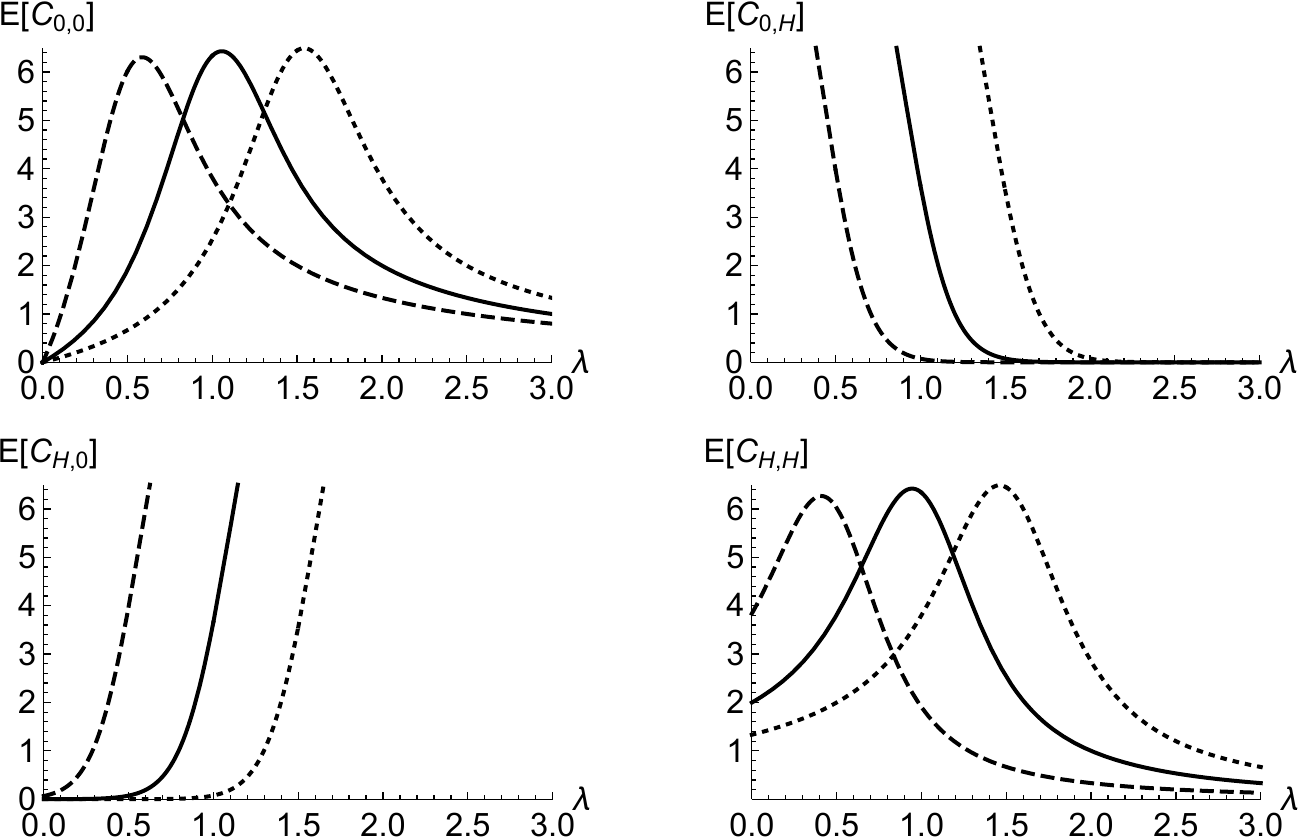}}
	\vfill
	\caption{The mean values of the renewal cycles as a function of $\lambda$, for $H$=10 with $\mu=0.5$ (dashed line), $\mu=1$ (solid line) 
	and $\mu=1.5$ (dotted line). }
	\label{fig:Cicli1}
\end{figure}
\begin{theorem}\label{teo3}
For the renewal cycles starting from the origin, for $\lambda\neq \mu$ the means are given by 
\begin{equation*}
\mathbb{E}(C_{0,0})=\frac{2 \lambda\left\{\lambda-\mu e^{2H(\mu-\lambda)}+(\mu-\lambda)e^{H(\mu-\lambda)}\left[1+H(\lambda+\mu)\right]\right\}}{(\lambda-\mu)\left(\lambda-\mu e^{H(\mu-\lambda)}\right)^2},
\end{equation*}
\begin{equation*}
\mathbb{E}(C_{0,H})=\frac{e^{H(\mu-\lambda)} 
\left\{4 \lambda \mu [-1+e^{H(\mu-\lambda)}]+H (\lambda^2-\mu^2) [\lambda +\mu e^{H(\mu-\lambda)}]\right\}}
{(\lambda-\mu)\left(\lambda-\mu e^{(\mu-\lambda)H}\right)^2}.
\end{equation*}
For the renewal cycles starting from $H$, for $\lambda\neq \mu$ the means are given by 
\begin{equation*}
\begin{aligned}
\mathbb{E}(C_{H,0})=\frac{
\mu e^{H(\lambda-\mu)} [\lambda (4+\lambda H)-\mu^2 H]+\lambda e^{2H (\lambda-\mu)} [-4 \mu+ H (\lambda^2-\mu^2)]}
{(\lambda-\mu)\left(\mu-\lambda e ^{H(\lambda-\mu)}\right)^2},
\end{aligned}
\end{equation*}
\begin{equation*}
\mathbb{E}(C_{H,H})=\frac{2 \mu\left\{\lambda-\mu e^{2H(\mu-\lambda)}+(\mu-\lambda)e^{H(\mu-\lambda)}\left[1+H(\mu+\lambda)\right]\right\}}{(\lambda-\mu)\left(\lambda-\mu e^{(\mu-\lambda)H}\right)^2}.
\end{equation*}
Moreover, in the special case $\lambda=\mu$, we have
\begin{equation*}
\mathbb{E}(C_{0,0})=\mathbb{E}(C_{H,H})=\frac{\mu H^2 (3+2 \mu H)}{3(1+\mu H)^2},
\end{equation*}
\begin{equation*}
\mathbb{E}(C_{0,H})=\mathbb{E}(C_{H,0})=\frac{H(3+3\mu H+\mu^2 H^2)}{3(1+\mu H)^2}.
\end{equation*}
\end{theorem}
\begin{proof}
The proof follows easily from Theorem \ref{teor1}, Proposition \ref{prop2} and recalling Eqs.\ \eqref{c00}, \eqref{c0h}, \eqref{defchh} and \eqref{defch0}.
\end{proof}
\par
Figures \ref{fig:Cicli1} and \ref{fig:Cicli2} provide some plots of the expected values of renewal cycles for different choices of the parameters. 
\par
\begin{remark}
From the first result of Theorem \ref{teo3} it is not hard to see that, for $\lambda>\mu$, 
$$
 \lim_{H\to \infty} \mathbb{E}(C_{0,0})=\frac{2}{\lambda-\mu},
$$
which is equal to the mean return time to the origin of the telegraph process in the presence of a single boundary at 0 
(cf.\ Proposition 3 of \cite{DiCrescenzo}). 
\end{remark}
\par
\begin{remark}\label{eq:Remark5.2}
Consider the process $\widetilde X_c(t)=c\,X(t)$, $t\geq 0$. Clearly, this is an asymmetric telegraph process with velocity $c$ and 
parameters $\lambda$, $\mu$. With reference to the stopping times $C_{0,v}$, $v\in \{0,H\}$, introduced in Section \ref{sec2}, 
we now denote by $\widetilde C_{0,v}$, $v\in \{0,H\}$, the corresponding stopping times for the process $\widetilde X_c(t)$. 
Since the hitting times of $\widetilde X_c(t)$ though $\{0, H\}$ correspond to the hitting times 
of $X(t)$ though $\{0, H/c\}$, we clearly have $\mathbb{E}(\widetilde C_{0,v})=\mathbb{E}(C_{0,v})\big|_{H=H/c}$, 
with $\mathbb{E}(C_{0,v})$, $v\in \{0,H\}$, given in Theorem \ref{teo3}. Hence, 
it is not hard to see that 
\begin{equation}
 \lim_{\sc SC}\mathbb{E}(\widetilde C_{0,v})=0, \qquad v\in \{0,H\},
 \label{eq:limitSC}
\end{equation}
where $\lim_{\sc SC}$ means that the limit is performed under the scaling conditions that allow $\widetilde X_c(t)$ to converge to the 
Wiener process with drift $\eta=\beta-\alpha$ and infinitesimal variance $\sigma^2$, i.e.\ (see Section 5 of L\'opez and Ratanov \cite{Lopez})
\begin{equation}
 c\to \infty, 
 \qquad 
 \lambda\sim \frac{1}{\sigma^2}\left(c^2+2\alpha c\right),
 \qquad 
 \mu  \sim \frac{1}{\sigma^2}\left(c^2+2\beta c\right), 
 \qquad \alpha,\beta>0.
 \label{eq:scaling}
\end{equation}
Let us define the stopping time 
\begin{equation}
 \widetilde\tau:=\inf\{t>0: \widetilde X_c(t)\in \{0,H\}\}\equiv \inf\left\{t>0:   X(t)\in \left\{0,\frac{H}{c}\right\}\right\}, 
 \qquad \widetilde X_c(0)=X(0)=0.
 \label{eq:deftau}
\end{equation}
Clearly, we have $\widetilde\tau=\min\{\widetilde C_{0,0},\widetilde C_{0,H}\}$ and thus, 
from (\ref{eq:limitSC}), $\lim_{\sc SC}\mathbb{E}(\widetilde\tau)=0$.  
This is in agreement with the analogous result for the above mentioned limiting Wiener process, 
for which the corresponding stopping time is equal to 0 w.p.\ 1, and is also confirmed by 
the analysis performed in Section 2.1 of Domin\'e \cite{Domine95}. 
The cases when the initial state is equal to the upper boundary can be treated similarly. 
\end{remark}
%
\begin{figure}[t]
\vspace*{-0.5cm}
\centering
{\includegraphics[width=10cm,height=6cm]{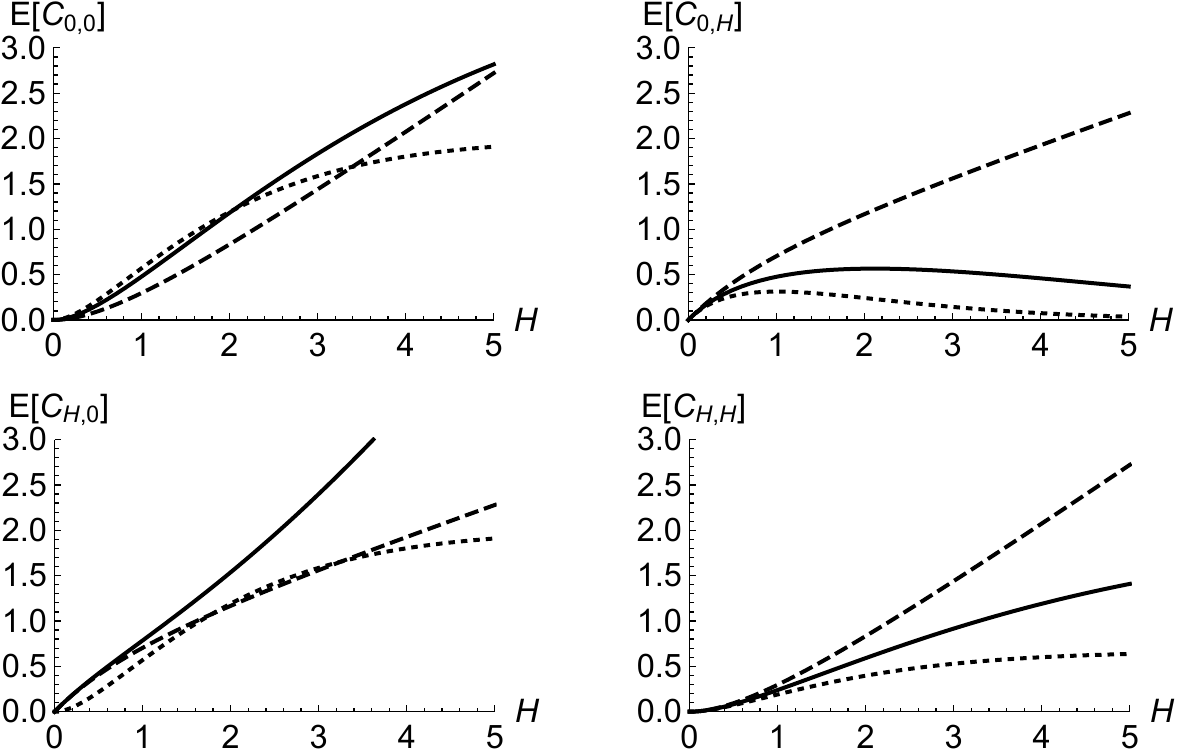}}
	\vfill
	\caption{The mean values of the renewal cycles given in Theorem \ref{teo3} as a function of $H$, 
	for $(\lambda,\mu)=(0.5,0.5)$ (dashed line), $(\lambda,\mu)=(1,0.5)$ (solid line) 
	and $(\lambda,\mu)=(1.5,0.5)$ (dotted line). }
	\label{fig:Cicli2}
\end{figure}
%
%
%
%
\section{Expected Time till Absorption}\label{Sec6}
Let us denote by $M$ the random variable counting the number of phases completed until absorption. Hence, recalling the model described in Section \ref{sec2}, $M$ has geometric distribution with parameter $\alpha\in(0,1)$, where $\alpha$ is the switching   
parameter, being the particles absorbed with probability $\alpha$, or reflected upwards (downwards) with probability $1-\alpha$.
Note that the random variables $M, C_{0,0}, C_{0,H}, C_{H,H}, C_{H,0}$ are mutually independent. 
\par
Let us denote by $L_n$ the expected length of a sample path leading to the absorption in the origin $0$ or in the level $H$ when $M=n$. 
We assume that $S_{0,0}$ is the sample path corresponding to Phase 1, $S_{0,H}$ the sample path for Phase 2, $S_{H,H}$ that for Phase 3, and $S_{H,0}$ the sample path corresponding to Phase 4. For example, if $M=1$ we have two possible sample paths, i.e.\ $S_{0,0}$ and $S_{0,H}$; if $M=2$ there are four possible sample paths: $\left\{S_{0,0}, S_{0,0}\right\}$, $\left\{S_{0,0}, S_{0,H}\right\}$, $\left\{S_{0,H}, S_{H,H}\right\}$ and $\left\{S_{0,H}, S_{H,0}\right\}$, and so on. 
Due to independence, the probability of each sample path is the product of the corresponding probabilities of phases within the sample path. So, the expected length of a sample path is the sum of the expected lengths of the corresponding phases. 
\par
In order to provide a formal expression of the expected lengths $L_n$, let us now introduce the matrix 
\begin{equation}
 {\bf P}=\left(
\begin{array}{cc}
 P_{0,0} & P_{0,H} \\
 P_{H,0} & P_{H,H} 
\end{array}
 \right)
\label{eq:Pmatr}
\end{equation}
and denote its $j$-th power by 
\begin{equation}
 {\bf P}^{(j)}=\left(
\begin{array}{cc}
 P_{0,0}^{(j)} & P_{0,H}^{(j)} \\
 P_{H,0}^{(j)} & P_{H,H}^{(j)}
\end{array}
 \right),
\label{eq:Pjmatr}
\end{equation}
for $j\in\mathbb{N}$. 
Moreover, let us now set 
\begin{equation}
Q^{(i,m)}_{u,v}=\sum_{j=i}^{m}P^{(j)}_{u,v},\qquad u,v\in\left\{0,H\right\},
\qquad i,m\in\mathbb N,\quad i\leq m.
\label{defQ}
\end{equation}
\begin{proposition}
The expected length of a sample path of $X(t)$ leading to the absorption in one of the boundaries 
during the  $n$-th phase can be expressed as
	\begin{equation}
	L_1=c_{0,0}P_{0,0}+c_{0,H}P_{0,H},
	\label{defL1}
	\end{equation} 
	and for $n\ge 2$
		\begin{equation}\label{th}
		\begin{aligned}
		L_n&=L_1\left[(1+P_{0,0})+P_{0,0}^{(2)}Q^{(0, n-3)}_{0,0}+P_{0,H}^{(2)}Q^{(1, n-3)}_{H,0}\right]\\
		&+L^*_1\left[P_{0,H}+P_{0,0}^{(2)}Q^{(1, n-3)}_{0,H}+P_{0,H}^{(2)}Q^{(0, n-3)}_{H,H}\right],
		\end{aligned}
		\end{equation}
where $P_{u,v}$ and $Q^{(i,m)}_{u,v}$, $u,v\in\left\{0,H\right\}$, are defined 
respectively in Eqs.\ (\ref{4}), (\ref{4bis}), and (\ref{defQ}), and where we have set 
$c_{u,v}=\mathbb E\left[C_{u,v}\right]$, for $u,v\in\{0,H\}$, and $L_1^*=c_{H,0}P_{H,0}+c_{H,H}P_{H,H}$.
\label{induz}
\end{proposition}
\begin{proof}
The expression of $L_1$ given in Eq.\ (\ref{defL1}) follows immediately from the definition. For $n\ge 2$ we prove the result by induction. Since
\begin{equation*}
L_2=2c_{0,0}P_{0,0}^2+(c_{0,0}+c_{0,H})P_{0,0}P_{0,H}+(c_{0,H}+c_{H,H})P_{0,H}P_{H,H}+(c_{0,H}+c_{H,0})P_{0,H}P_{H,0}, 
\end{equation*}
and being $P_{i,0}+P_{i,H}=1$ for $i\in\{0,H\}$, Eq.\ \eqref{th} holds for $n=2$. Noting that, for all $n\ge 2$, it is
$$
L_n=L_{n-1}+L_1P_{0,0}^{(n-1)}+L^*_1P_{0,H}^{(n-1)},
$$
and making use of the induction hypothesis, we obtain
\begin{equation*}
\begin{aligned}
	L_{n+1}&=L_n+L_1P_{0,0}^{(n)}+L^*_1P_{0,H}^{(n)}\\
	&=L_2+L_1P_{0,0}^{(2)}Q^{(0, n-3)}_{0,0}+L_1P_{0,H}^{(2)}Q^{(1, n-3)}_{H,0}+L^*_1P_{0,0}^{(2)}Q^{(1, n-3)}_{0,H}+L^*_1P_{0,H}^{(2)}Q^{(0, n-3)}_{H,H}\\
	&+L_1\left[P_{0,0}^{(2)}P_{0,0}^{(n-2)}+P_{0,H}^{(2)}P_{H,0}^{(n-2)}\right]+L^*_1\left[P_{0,0}^{(2)}P_{0,H}^{(n-2)}+P_{0,H}^{(2)}P_{H,H}^{(n-2)}\right]\\
	&=L_2+L_1P_{0,0}^{(2)}Q^{(0, n-2)}_{0,0}+L_1P_{0,H}^{(2)}Q^{(1, n-2)}_{H,0}+L^*_1P_{0,0}^{(2)}Q^{(1, n-2)}_{0,H}+L^*_1P_{0,H}^{(2)}Q^{(1,n-2)}_{H,H}\!\!,
\end{aligned}
\end{equation*}
so that the proposition immediately follows.
\end{proof}
\par
Let us now obtain the formal expression of the $j$-th power of matrix (\ref{eq:Pmatr}). 
\begin{proposition}\rm
The matrix  (\ref{eq:Pjmatr}) can be expressed as 
\begin{equation}
 {\bf P}^{(j)}=\left(
\begin{array}{cc}
\frac{P_{H,0}+P_{0,H}\vartheta^j }{P_{H,0}+P_{0,H}} & \frac{P_{0,H}-P_{0,H}\vartheta^j }{P_{H,0}+P_{0,H}} \\[0.4cm]
\frac{P_{H,0}-P_{H,0}\vartheta^j }{P_{H,0}+P_{0,H}} & \frac{P_{0,H}+P_{H,0}\vartheta^j }{P_{H,0}+P_{0,H}} 
\end{array}
 \right),
\label{eqPj2}
\end{equation}
where we have set
\begin{equation}
 \vartheta=\frac{P_{0,0}-P_{0,H}+P_{H,H}-P_{H,0}}{2}\equiv P_{0,0}+P_{H,H}-1.
 \label{defvartheta}
\end{equation}
\end{proposition}
\begin{proof}
Recalling the probabilities obtained in Theorem \ref{teor1}, it is easy to note that 
the matrix
	\begin{equation*}
	\bold P=\begin{pmatrix}
	P_{0,0}  &P_{0,H} \\ 
	P_{H,0} &P_{H,H} 
	\end{pmatrix}
	\end{equation*} 
has characteristic polynomial  given by
$p(\lambda)=(\lambda-\vartheta)(\lambda-1)$,
with $\vartheta$ given in (\ref{defvartheta}). 
For the sake of non triviality, we consider the case of $0<P_{u,v}<1$ with $u,v\in\{0,H\}$ and thus $\bold P$ 
admits of the following factorization: 
	$$
	{\bf P}={ \bf B}\cdot \left(
	\begin{array}{cc}
	\vartheta &0 \\
	0 & 1
	\end{array}
	\right) \cdot { \bf B}^{-1},
	$$
	where 
	$$
	{\bf B}=\left(
	\begin{array}{cc}
	-\frac{P_{0,H}}{P_{H,0}} &1 \\
	1 & 1
	\end{array}
	\right).
	$$
Hence, noting that 
$$
 {\bf B}^{-1}=\left(
\begin{array}{cc}
 -\frac{P_{H,0}}{P_{H,0}+P_{0,H}} & \frac{P_{H,0}}{P_{H,0}+P_{0,H}} \\[0.4cm]
\frac{P_{H,0}}{P_{H,0}+P_{0,H}} & \frac{P_{0,H}}{P_{H,0}+P_{0,H}}
\end{array}
 \right),
$$
Eq.\ (\ref{eqPj2}) follows from (\ref{eq:Pjmatr}) by means of straightforward calculations.
\end{proof}
%
%
\begin{figure}[t]
	\centering
	{\includegraphics[scale=0.8]{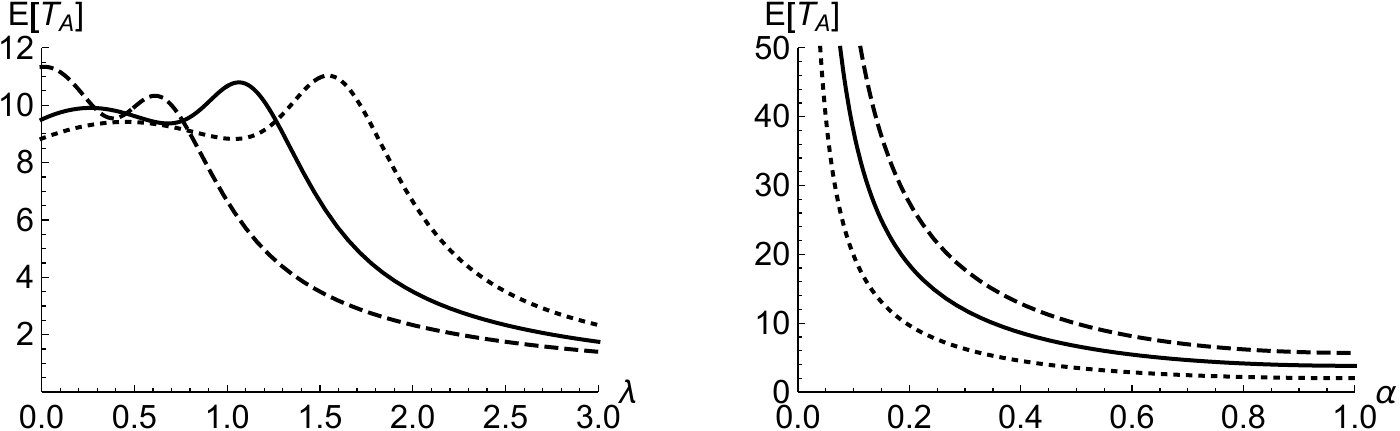}}
	\caption{Left: $\mathbb E(T_ A)$ as a function of $\lambda$ for $\alpha=0.5$ and $\mu=0.5$ (dashed line), $\mu=1$ (solid line), $\mu=1.5$ (dotted line). 
	Right: $\mathbb E(T_ A)$ as a function of $\alpha$ for $\mu=0.5$ and $\lambda=0.5$ (dashed line), $\lambda=1$ (solid line), $\lambda=1.5$ (dotted line). In both cases $H=10$.}
	\label{fig:FiguraMediaTa}
\end{figure}
\par
Let us now denote by ${T_ A}$ the time till absorption in the origin or in the boundary $H$, 
i.e.
\begin{equation}\label{defTA}
 T_ A=\inf\{s>0: X(t)\in\{0,H\} \quad \forall \;t>s\}. 
\end{equation}
In the next theorem we provide its expected value.
\begin{theorem}
For $0<\alpha<1$, the mean time till absorption in the origin or in the boundary $H$ is 
\begin{equation}\label{timeabsorption}
\begin{aligned}
\mathbb E({T_ A})&=L_1\left[1-\alpha(1-\alpha)+(1-\alpha)P_{0,0}+(1-\alpha)^2\frac{\left[1-P_{0,H}(2-P_{0,H}-P_{H,0})\right]}{\alpha+(1-\alpha)(P_{0,H}+P_{H,0})}\right.\\
&\left.+\frac{(1-\alpha)^3}{\alpha}\frac{P_{H,0}}{\alpha+(1-\alpha)(P_{0,H}+P_{H,0})}\right]+L^*_1P_{0,H}(1-\alpha)\\
&\times \left[\alpha +(1-\alpha)\frac{2+(1-P_{0,H}-P_{H,0})\alpha}{\alpha+(1-\alpha)(P_{0,H}+P_{H,0})}+\frac{(1-\alpha)^2}{\alpha}\frac{1}{\alpha+(1-\alpha)(P_{0,H}+P_{H,0})}\right], 
\end{aligned}
\end{equation}
where the probabilities $P_{0,0}$, $P_{0,H}$, $P_{H,0}$, $P_{H,H}$ and the values of $L_1$, $L_1^*$ have been obtained in Theorem \ref{teor1} and Proposition \ref{induz}, respectively.
\end{theorem}
\begin{proof}
Since 
\begin{equation*} 
\mathbb E({T_ A})=\alpha\sum_{n=1}^{+\infty}L_n(1-\alpha)^{n-1},
\end{equation*}
recalling the expressions of $L_2$ and $L_n$,  we have
\begin{equation*}
\begin{aligned}
\mathbb E({T_A})&=\alpha L_1+\alpha(1-\alpha)L_1P_{0,0}+\alpha(1-\alpha)L^*_1P_{0,H}+ L_1(1+P_{0,0})(1-\alpha)^2\\
&+\alpha L_1P_{0,0}^{(2)}\sum_{n\ge 3}(1-\alpha)^{n-1}Q^{(0, n-3)}_{0,0}+\alpha L_1P_{0,H}^{(2)}\sum_{n\ge 4}(1-\alpha)^{n-1}Q^{(1, n-3)}_{H,0}\\
&+L^*_1P_{0,H}(1-\alpha)^2+L^*_1\alpha P_{0,H}^{(2)}\sum_{n\ge 3}(1-\alpha)^{n-1}Q^{(0, n-3)}_{H,H}+\\
&+L^*_1\alpha P_{0,0}^{(2)}\sum_{n\ge 4} (1-\alpha)^{n-1}Q^{(1, n-3)}_{0,H}.
\end{aligned}
\end{equation*}
Hence,
\begin{equation}\label{14}
\begin{aligned}
\mathbb E({T_ A})&=L_1\left[\alpha +\alpha(1-\alpha)P_{0,0}+(1+P_{0,0})(1-\alpha)^2\right]\\
&+L^*_1\left[\alpha(1-\alpha)P_{0,H}+P_{0,H}(1-\alpha)^2\right]\\
&+L_1P_{0,0}^{(2)}\sum_{j=0}^{+\infty}(1-\alpha)^{j+2}P_{0,0}^{(j)}+L_1P_{0,H}^{(2)}\sum_{j=1}^{+\infty}(1-\alpha)^{j+2}P_{H,0}^{(j)}\\
&+L^*_1P_{0,H}^{(2)}\sum_{j=0}^{+\infty}(1-\alpha)^{j+2}P_{H,H}^{(j)}+L^*_1P_{0,0}^{(2)}\sum_{j=1}^{+\infty}(1-\alpha)^{j+2}P_{0,H}^{(j)}.
\end{aligned}
\end{equation}
Finally, recalling Eq. (\ref{eqPj2}), from Eq.\ \eqref{14} we obtain
\begin{equation*}
\begin{aligned}
\mathbb E({T_ A})&=L_1\left[\alpha +\alpha(1-\alpha)P_{0,0}+(1+P_{0,0})(1-\alpha)^2\right]\\
&+L^*_1\left[\alpha(1-\alpha)P_{0,H}+P_{0,H}(1-\alpha)^2\right]\\
&+L_1P_{0,0}^{(2)}\frac{(1-\alpha)^2}{P_{0,H}+P_{H,0}}\left[\frac{P_{H,0}}{\alpha}+\frac{P_{0,H}}{1-\vartheta(1-\alpha)}\right]\\
&+L_1P_{0,H}^{(2)}\frac{(1-\alpha)^3P_{H,0}}{P_{0,H}+P_{H,0}}\left[\frac{1}{\alpha}-\frac{\vartheta}{1-\vartheta(1-\alpha)}\right]\\
&+L^*_1P_{0,H}^{(2)}\frac{(1-\alpha)^2}{P_{0,H}+P_{H,0}}\left[\frac{P_{0,H}}{\alpha}+\frac{P_{H,0}}{1-\vartheta(1-\alpha)}\right]\\
&+L^*_1P_{0,0}^{(2)}\frac{(1-\alpha)^3P_{0,H}}{P_{0,H}+P_{H,0}}\left[\frac{1}{\alpha}-\frac{\vartheta}{1-\vartheta(1-\alpha)}\right],
\end{aligned}
\end{equation*}
so that the thesis immediately follows.
\end{proof}
%
\begin{figure}[t]
	\centering
	{\includegraphics[scale=0.8]{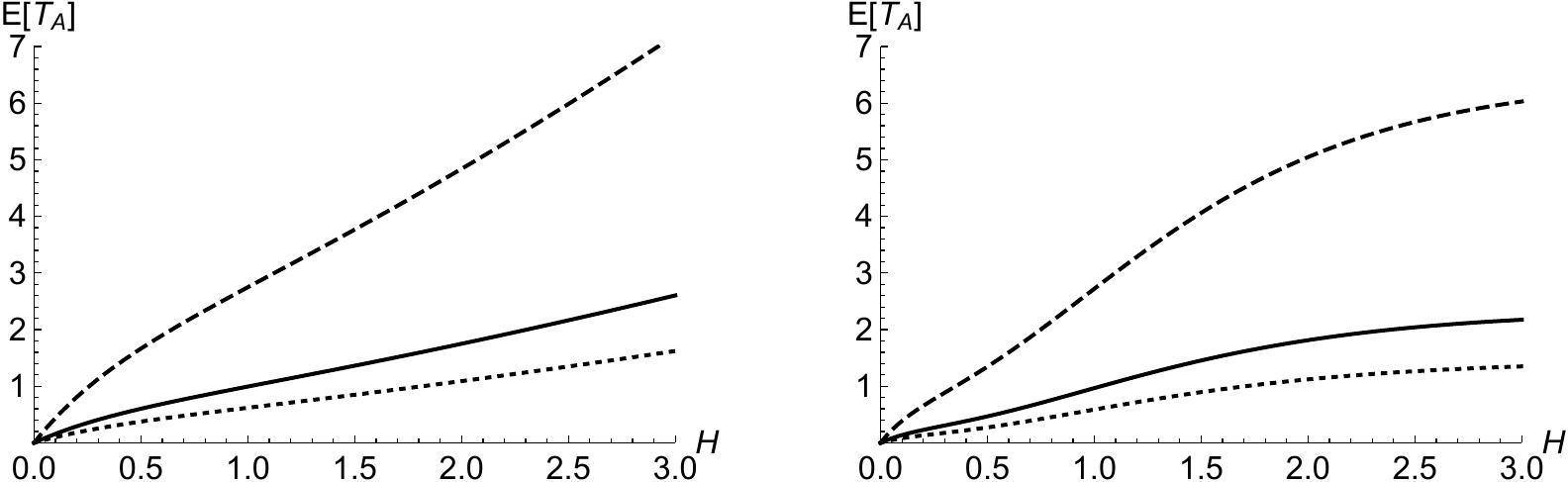}}
	\caption{$\mathbb E(T_ A)$ as a function of $H$ for $\alpha=0.2$ (dashed line), $\alpha=0.5$ (solid line), $\alpha=0.8$ (dotted line), with 
	$\lambda=\mu=0.5$ (left) and  $(\lambda,\mu)=(2,0.5)$ (right).}
	\label{fig:FiguraMediaTa4}
\end{figure}

The mean absorbing time to one of the boundaries, 
$\mathbb E(T_A)$, is plotted for some choices of the parameters in Figures \ref{fig:FiguraMediaTa} and \ref{fig:FiguraMediaTa4}. 
As we expect, it is increasing as the upper endpoint $H$ increases. 
Moreover, it decreases when the switching    
probability $\alpha$ increases. In particular, $\mathbb E(T_A)$ diverges 
when $\alpha\rightarrow 0^{+}$, whereas, if $\alpha\rightarrow 1$, it approaches the value $c_{0,0} P_{0,0}+c_{0,H} P_{0,H}\equiv L_1$. The latter  is the expected length 
of the sample path leading to the absorption in the origin or in the level $H$ in the case $M=1$, i.e.\ when 
the absorption occurs at the first hitting of one of the boundaries. 
Moreover, $\mathbb E(T_A)$  is non-monotonic in $\lambda>0$; it tends to zero as $\lambda$ goes to $+\infty$, whereas it tends to a finite value when 
$\lambda$ approaches $0^+$.
\par
In addition, similarly as in Remark \ref{eq:Remark5.2}, with reference to the stopping time introduced in (\ref{defTA}), we 
denote by $\widetilde T_A$ the corresponding stopping time for the transformed process $\widetilde X_c(t)$. 
It is not hard to see that, under scaling (\ref{eq:scaling}), making use of (\ref{timeabsorption})  one can obtain 
$\lim_{\sc SC}\mathbb{E}(\widetilde T_A)=0$. Again, this  is in agreement with the analogous natural result 
for the limiting Wiener process in the presence of two boundaries of the same nature. 
%
\section{Conclusions}
%
The study of diffusion processes constrained by boundaries has a long history. 
The approach usually adopted is based on the resolution of partial differential equations
with suitable boundary conditions. On the contrary, stochastic processes describing 
finite velocity random motions in the presence of boundaries have not been studied extensively. 
Over the years, some results have been obtained on the telegraph process subject to reflecting or 
absorbing boundaries, whereas the case of  
hard reflection at the boundaries (with random switching to full absorption) 
seems, to our knowledge, quite new. Along the line of a previous paper \cite{DiCrescenzo} 
concerning the one-dimensional telegraph process under a single boundary, 
in the present contribution we investigate the case of a random motion confined by two boundaries 
of the above described type. 
The main results obtained here are related to the expected values of the renewal cycles and 
of the time till the absorption. 
The novelty of the adopted approach, which is based on the study of first-passage times of a suitable 
compound Poisson process, is a strength of the paper. 
%
\subsection*{Acknowledgements}
This work is partially supported by the group GNCS of INdAM (Istituto Nazionale di Alta Matematica), 
and by MIUR - PRIN 2017, project `Stochastic Models for Complex Systems', no. 2017JFFHSH. 
We thank an anonymous referee for useful comments that allowed us to improve the paper. 
%

\end{document}